\documentclass[12pt,reqno]{article}

\usepackage{amsthm,amsmath,amssymb}
\usepackage{appendix}

\usepackage[colorlinks=true,citecolor=black,linkcolor=black,urlcolor=blue]{hyperref}

\usepackage[usenames,dvipsnames]{color}
\usepackage{amscd}
\usepackage[english]{babel}
\usepackage{graphics}
\usepackage{amsfonts}
\usepackage{latexsym}
\usepackage{float}
\usepackage{bm}
\usepackage{enumerate}
\usepackage[footnotesize]{caption}

\usepackage{rotating}
\usepackage{pdflscape}

\newcommand{\NN}{{\mathbb N}} 
\newcommand{\Snj}{\scriptstyle} 
\DeclareMathOperator{\Gen}{Gen} 
\DeclareMathOperator{\Ht}{ht} 
\DeclareMathOperator{\Om}{\omega} 
\DeclareMathOperator{\DR}{\searrow} 
\DeclareMathOperator{\DL}{\swarrow} 
\newcommand{\K}{K} 
\newcommand{\B}{E} 
\newcommand{\Bbi}{\bm{E}} 
\newcommand{\m}{m} 
\newcommand{\x}{x} 
\newcommand{\y}{y} 
\newcommand{\sI}{\mathcal{I}} 
\newcommand{\cB}{\mathcal{B}} 
\newcommand{\F}{F} 
\newcommand{\seqnum}[1]{\href{http://oeis.org/#1}{\underline{#1}}} 
\newcommand{\beql}[1]{\begin{equation}\label{#1}} 
\newcommand{\eeq}{\end{equation}} 
\newcommand{\al}{\alpha} 
\newcommand{\la}{\lambda} 
\newcommand{\Floor}[1]{\left\lfloor#1\right\rfloor}
\newcommand{\Ceil}[1]{\left\lceil#1\right\rceil}

\newtheorem{thm}{Theorem}{\bfseries}{\itshape}
\newtheorem{cor}[thm]{Corollary}{\bfseries}{\itshape}
\newtheorem{lem}[thm]{Lemma}{\bfseries}{\itshape}
{\bfseries}{\itshape}
\newtheorem{conj}[thm]{Conjecture}{\bfseries}{\itshape}

\newlength{\commentlength}
\begin{document}
\theoremstyle{plain}

\begin{center}
{\large\bf On Kaprekar's Junction Numbers } \\
\vspace*{+.2in}

Max A. Alekseyev \\
The George Washington University, \\
Department of Mathematics, \\
Washington, DC 20052, USA \\
Email: maxal@gwu.edu \\

\vspace*{+.1in}

and \\

\vspace*{+.1in}

N. J. A. Sloane\footnote{Corresponding author.} \\
The OEIS Foundation, Inc., \\
11 So. Adelaide Ave., \\
Highland Park, NJ 08904, USA \\
Email: njasloane@gmail.com \\

\vspace*{+.2in}

{\bf Abstract}

\end{center}

A base $b$ junction number $u$ has the property that there are
at least two ways to write it as $u = v + s(v)$,
where $s(v)$ is the sum of the digits in the expansion
of the number $v$ in base $b$.
For the base $10$ case, Kaprekar in the 1950's and 1960's studied the
problem of finding $\K(n)$, the smallest $u$ such that
the equation $u=v+s(v)$ has exactly $n$ solutions.
He gave the values $\K(2)=101$, $\K(3)=10^{13}+1$,
and conjectured that $\K(4)=10^{24}+102$.
In 1966 Narasinga Rao gave the upper bound $10^{1111111111124}+102$
for $\K(5)$, as well as upper bounds for
$\K(6)$, $\K(7)$, $\K(8)$, and $\K(16)$.
In the present work, we derive a set of recurrences which determine $\K(n)$ for any base $b$  
and in particular imply that these conjectured values of $\K(n)$ are correct.
The key to our approach is an apparently new recurrence for $F(u)$, the number of solutions to $u=v+s(v)$.
We illustrate our method by computing the values of $\K(n)$ for $n\leq 16$ and bases $b \le 10$, and show that for each base $\K(n)$ grows as a tower of height proportional to $\log_2(n)$.
Rather surprisingly, the values of $\K(n)$ for the base $5$ problem are determined by the classical Thue--Morse sequence, which leads us to define generalized Thue--Morse sequences for other bases.


\noindent  \emph{Keywords: } 
Junction numbers, self-numbers, Colombian numbers, Kaprekar, 
Thue--Morse sequence.

\vspace{0.1in}
\noindent  2010 {\it Mathematics Subject Classification}:
Primary 11A63; Secondary: 11B37.



\section{Introduction}\label{Sec1}
For a fixed base $b \ge 2$, let $s(v)$ denote the
sum of the digits in the base $b$ expansion of 
$v \in \NN = \{0, 1, 2, \ldots\}$,
and let $f(v):=v+s(v)$.
Sequences that arise by iterating $f$ have a long history
\cite{Agr14,Sto76} (the latter 
reference has an extensive bibliography).
In the 1950's and 1960's, 
Dattaraya Ramchandra Kaprekar
in a series of self-published booklets~\cite{Kap51,Kap59,Kap62,Kap63,Kap67}
studied the inverse mapping to $f$ in the base $10$ case.
Let $\Gen(u) := \{ v \in \NN \mid f(v)=u \}$ and $F(u) := |\Gen(u)|$.
Kaprekar called the elements of $\Gen(u)$ the {\em generators} of $u$,
and in 1956~\cite{Kap56} defined a {\em self-number} to be any number $u$ with $F(u)=0$.
The first few self-numbers (in base $10$) are
\beql{EqSelf}
1,3,5,7,9,20,31,42,53,64,75,86,97,108,110,121,132,143,154,165,176, \ldots 
\eeq
(\seqnum{A003052}\label{A003052}).\footnote{Throughout this article,
six-digit numbers prefixed by A refer to entries in the OEIS \cite{OEIS}.}
Self-numbers are also known as {\em Colombian numbers},
after a problem proposed by Recam\'{a}n in 1973 \cite{Rec73}.

Kaprekar called numbers with at least two generators {\em junction
numbers}. The smallest junction number (again in base $10$) is $101$, which
has generators 91 and 100, and the first few junction numbers are
$$
101,103,105,107,109,111,113,115,117,202,204,206,208,210,212,214, \ldots
$$
(\seqnum{A230094}\label{A230094}).

Kaprekar was particularly interested in finding what we will call $\K(n)$,
the smallest number with $n$ generators,  which is the 
subject of the present paper.
We will show that
the sequence $\left(\K(n)\right)_{n\geq 1}$ begins
\begin{align}\label{EqKb10}
0, ~101, & ~10^{13}+1, ~10^{24}+102,
~10^{1111111111124} + 102,
~10^{2222222222224} + 10^{13}+2, \nonumber \\
& 10^{ (10^{24} + 10^{13} + 115) / 9 } + 10^{13} + 2,
~10^{ (2\cdot 10^{24} + 214)/9 } + 10^{24} + 103, \nonumber \\
& 10^{(10^{1111111111124}+10^{24}+214)/9}+10^{24}+103,  ~\ldots
\end{align}
(\seqnum{A006064}\label{A006064}).\footnote{The reader may detect a pattern in these 
numbers, but should be warned that it breaks down after a while. 
See Tables~\ref{TabK2} and \ref{TabK10c}.}
It is easy to check by hand that $\K(2)=101$,
and with today's computers it is easy to verify $\K(3)=10^{13}+1$ 
by direct search.
As to what was known by Kaprekar and his colleagues more than fifty
years ago, the various accounts
given by Kaprekar \cite{Kap62},
Narasinga Rao \cite{NRao66}, and
Gardner \cite{Gar88} do not quite agree, but it seems
that Kaprekar believed that he had proved that $\K(3)=10^{13}+1$,
that the putative value $10^{24}+102$ for $\K(4)$ was discovered 
independently by Kaprekar and Professor Gunjikar in 1961, and that Kaprekar
was convinced that it was the true value of $\K(4)$ and not just an upper bound.
Gardner \cite{Gar88} reports in 1975 that Kaprekar told him that he
had also found what he conjectured to be the values of $\K(5)$ and $\K(6)$.
Kaprekar's work on this problem is also discussed 
by Schorn \cite{Sch04}.

However, Narasinga Rao, writing in 1963 \cite{NRao66} (although not
published until 1966),
states things slightly differently. He gives a recipe for
finding junction numbers with a specified number of
generators, improving on an earlier recipe of Kaprekar's, 
and gives Kaprekar's value of $\K(3)=10^{13}+1$.
He then {\em conjectures} that $\K(4)=10^{24}+102$,
and gives as a candidate for $\K(5)$ the value
$10^{1111111111124}+102$, remarking that no much smaller 
value is likely to exist. (Narasinga Rao's recipe does not necessarily
produce the smallest junction number with a given number of generators.) 
He also gives upper bounds for $\K(6)$, $\K(7)$, $\K(8)$, and $\K(16)$.
Remarkably enough, all of Narasinga Rao's upper bounds turn out
to be the true values for these $\K(n)$.
We return to the base $10$ case in Section~\ref{SecB10}.


\

The principal goal of this paper is to present a set of recurrences
which generate the sequence $\left(\K(n)\right)_{n\geq 1}$ for any base $b$
(see Section~\ref{SecBb}, and in particular Theorems~\ref{ThBB} and~\ref{ThS2}).
These recurrences depend upon an apparently new recurrence
for $F(u)$, discussed in Section~\ref{SecRec}.

Because the values of $\K(n)$ for 
small $b$ and $n$ are both easy to determine
and somewhat exceptional,
we start by discussing the cases $n \le 3$ in
Section~\ref{SecPropK} and bases $b=2$ and $b=3$ in 
Section~\ref{SecB2}. 
Bases $2$ and $3$ are also exceptional since for them
the recurrences for $\K(n)$ are quite simple
and can be obtained without the machinery 
developed in Section~\ref{SecBb}.
Tables~\ref{TabK1} and \ref{TabK2} in Section~\ref{SecRec} 
collect the numerical values of $\K(n)$ for $n \le 7$
and bases $b \le 10$.

Section~\ref{SecBb} gives the recurrences for a general base $b$.
Sections~\ref{SecB5} and \ref{SecB10} apply the results
of Section~\ref{SecBb} to bases $b\in\{4, 5, 7, 10\}$.
In general, the calculation of $\K(n)$ involves a subsidiary sequence
$\left(\tau(n)\right)_{n \ge 1}$
of integers in the range $[0,b-2]$ when $b$ is even, or
$[0,\tfrac{b-3}{2}]$ when $b$ is odd.
For $b=5$, it turns out that $\left(\tau(n)\right)_{n \ge 1}$ is essentially
the classical Thue--Morse sequence\footnote{This provides yet another illustration of the ubiquity of this sequence~\cite{AS99}.} \seqnum{A010060}\label{A010060},
and so, for any base, we refer to $\left(\tau(n)\right)_{n \ge 1}$
as a ``generalized Thue--Morse sequence".
For example, for both bases $b=4$ and $b=7$ we obtain the ternary sequence
shown in~\eqref{EqTM4}, and for base $b=10$ the sequence~\eqref{EqTM10}.
 
Section~\ref{SecTow} (and in particular Conjecture~\ref{ConjTow})
discusses the rate of growth of $\K(n)$ as a function of $n$.
We begin by applying the recurrence for $F(u)$
from Section~\ref{SecRec} to establish some general bounds on $\K(n)$.
Then we consider the representation of $\K(n)$ as a tower of exponentials.
For bases $b \ne 3$, it appears that $\K(n)$ is a tower
\beql{EqTow10}
\K(n) ~=~ 
b^{ b^{\Snj \cdot ^{ \Snj \cdot ^{\Snj \cdot ^{\Snj b ^{\Snj \Om(n) }}}}}}\,,
\eeq
with $0 < \Om(n) \le 1$, of height\footnote{The height of a tower is defined in \eqref{EqTowU} in Section~\ref{SecTow}.} 
$\Ceil{\log_2(n)} + \la$,
where $\la = 3$ if $b=2$, $\la = 2$ if $b \ge 4$ is even, and
$\la = 1$ if $b \ge 5$ is odd (base $b=3$ is slightly exceptional).

\vspace{0.1in}
\noindent \textbf{Notation.} We will always work in a fixed base $b \ge 2$. 
If the base $b$ expansion of  $v \in \NN := \{0, 1, 2, \ldots\}$ is
\beql{Eq1_1}
v ~=~ \sum_{i=0}^{k-1} v_i b^i
\eeq
(where $0 \le v_i <b$, $k = \Ceil{\log_b (v+1)}$, 
and $v_{k-1} \ne 0$ unless $v=0$),
we refer to the $v_i$ as ``digits", even if $b \ne 10$,
and say that $v$ has ``length" $k$.
We will also write $v$ as
\beql{Eq1_2}
v ~=~ [v_{k-1}, v_{k-2}, \ldots, v_1, v_0] _b\,,
\eeq
where we omit the commas between the digits
if there is no possibility of confusion.
In this notation, $s(v) = \sum_{i=0}^{k-1} v_i$ and 
$f(v)=\sum_{i=0}^{k-1} v_i (b^i+1)$.

As already mentioned, 
for $u \in \NN$, we let $\Gen(u) := \{ v \in \NN \mid f(v)=u \}$, 
$F(u) := |\Gen(u)\}|$, and for $u<0$ we set
$\Gen(u) := \emptyset$ (the empty set) and $F(u):=0$.
For $n \ge 1$, $\K(n)$ is defined to be the smallest 
$u \in \NN$ such that $F(u)=n$.
For $n \ge 2$, it will turn out that the leading base-$b$ digit 
of $\K(n)$ is $1$ (see Theorem~\ref{Th131_1}),
and we define $\B(n)$ for $n \ge 2$ by 
$\K(n) = b^{\B(n)} +$  terms of smaller order 
(in other words, $\B(n)+1$ is the length of $\K(n)$ in base $b$).
In Tables~\ref{TabK1} and~\ref{TabK2} we write $\B_b(n)$ to
indicate the value of $b$. 
For use in Section~\ref{SecBb},
in Section~\ref{SecPropK} we also define 
$\K_i(n)$ to be the smallest $u \in \NN$ such that 
$F(u)=n$ and $u \equiv i \pmod{b-1}$.
Of course, the functions $s$, $f$, $F$,  $\Gen$, $\K$, $\K_i$
all depend on the value of $b$, but to indicate this with subscripts
would have made the equations unnecessarily complicated,
so we hope the value of $b$ will always be clear from the context.

In summary, the principal symbols are
$$
\begin{array}{lcl}
b & : &\text{base}, \ge 2, \nonumber \\
s(v) & : &\text{sum of base-$b$ digits of $v$}, \nonumber \\
f(v) & : &v+s(v), \nonumber \\
\Gen(u) & : &\text{set of $v$ such that~} f(v)=u, \nonumber \\
F(u) & : &\text{number of $v$ such that~} f(v)=u, \nonumber \\
\K(n) & : &\text{smallest $u$ such that~} F(u)=n, \nonumber \\
\K_i(n) & : & \text{smallest $u$ such that~} F(u)=n \text{~and~} u\equiv i\pmod{b-1}, \nonumber \\
\B(n)  & : & \text{for} \, n \ge 2, \K(n) = b^{\B(n)} + \textrm{terms of smaller order}, \nonumber \\
\sI & : & \text{subset of the residue classes modulo $b-1$ defined by~}\eqref{sIdef}, \nonumber \\
J(n) & : & \text{subset of $\sI$ defined by~}\eqref{EqJn},
\end{array}
$$
where $b,u,v,n \in \NN$.



\section{Preliminary results}\label{SecLem}


We begin with some elementary lemmas.
\begin{lem}\label{Lem1}
\begin{enumerate}[(i)]
    \item If $b$ is odd, then $f(v)$ is even for any $v \in \NN$,
and so $F(u)=0$ if $u$ is odd.

\item If $b$ is even, then $F(u)=0$ if $u$ is odd and $u<b$.
\end{enumerate}
\end{lem}

\begin{proof}
(i) If $b$ is odd and $v$ is given by \eqref{Eq1_1} then 
$f(v) = \sum_{i} v_i (b^i+1)$ is even. 

(ii) Numbers below $b$ have at most one generator, $v$ (say),
for which $f(v)=2v$ is even.
\end{proof}

The next lemma says that if $u=f(v)$, then
$v$ is smaller than $u$, but not too much smaller.
This is useful when making computer searches.


\begin{lem}\label{Lem2}
If $u \ge 2, v \in \NN$ satisfy $v+s(v)=u$, then
\beql{Eq107_3}
u- (b-1)\Ceil{\log_b (u)} \le v \le u-1 \,.
\eeq
\end{lem}
\begin{proof}
Since $u \ge 2$, $v \ge 1$, $s(v) \ge 1$, and so $v \le u-1$. Since
the length of $v$ is 
$k = \Ceil{\log_b(v+1)} \le \Ceil{\log_b (u)}$,
we have $s(v) \le (b-1)\Ceil{\log_b (u)}$.
\end{proof}

The third lemma is a generalization of the observation
that the Hamming weight of $2^{\m}-1-v$ is equal to 
${\m} - \text{Hamming~weight}(v)$. We omit the proof.


\begin{lem}\label{Lem3}
If  ${\m} \ge 0$, $1 \le c \le b-1$, and $0 \le v \le cb^{\m}-1$, 
then
\beql{Eq112_1}
s(cb^{\m}-1-v) ~=~ (b-1){\m}+c-1-s(v)\,.
\eeq
\end{lem}
\noindent
For example, in base $10$, $s(281) = s(3\cdot10^2-1-18) = 9\cdot2+2-9 = 11$.


\begin{lem}\label{Lem4} Let $n\geq 2$ be an integer.
\begin{enumerate}[(i)]
    \item 
If $\left(a(i)\right)_{i \ge 1}$ is a sequence of nonnegative real numbers such that
$a(m+1)\geq 2a(m)$ for all $m=1,2,\dots,n-2$,
then
\beql{EqSplit}
\min_{1 \le i \le n-1} a(i)+a(n-i) ~=~
a\left( \Ceil{ \tfrac{n}{2} } \right) +
a\left( \Floor{ \tfrac{n}{2} } \right) 
 \quad \text{for~} n \ge 2\,.
\eeq
Moreover, if $a(2)>0$ then $i\in\left\{\Floor{ \tfrac{n}{2} }, \Ceil{ \tfrac{n}{2} }\right\}$ 
are the only values of $i$ that attain the minimum in~\eqref{EqSplit}.

\item If $\left(a_1(i)\right)_{i \ge 1}$ and $\left(a_2(i)\right)_{i \ge 1}$ are a pair of
sequences of nonnegative real numbers such that
\beql{EqSplit2}
a_1(m+1)\geq a_1(m)+a_2(m) \quad \text{and} \quad a_2(m+1) \geq a_1(m)+a_2(m)\,
\eeq
for all $m=1,2,\dots,n-2$,
then
\begin{align}\label{EqSplit3}
\min_{1 \le i \le n-1} & a_1(i)+a_2(n-i) \nonumber \\
& ~=~ \min \left\{
a_1\left( \Ceil{ \tfrac{n}{2} } \right) +
a_2\left( \Floor{ \tfrac{n}{2} } \right),
a_1\left( \Floor{ \tfrac{n}{2} } \right) +
a_2\left( \Ceil{ \tfrac{n}{2} } \right)\right\}\,.
\end{align}
Moreover, if $a_1(2)>0$ and $a_2(2)>0$ then 
$i=\Floor{ \tfrac{n}{2} }$ and $i=\Ceil{ \tfrac{n}{2} }$ 
are the only values of $i$ that attain the
minimum in the left-hand side of~\eqref{EqSplit3}.
\end{enumerate}
\end{lem}

\begin{proof}
(i) Suppose $n=2t$. Then
$2a(t) \leq a(t-1)+2a(t) \leq a(t-1)+a(t+1) \leq 2a(t+1) 
\leq a(t+2) \leq a(t-2)+a(t+2) \leq \cdots \leq a(1)+a(n-1)$.
If $n=2t+1$ is odd, then $a(t)+a(t+1) \leq 2a(t+1) \leq a(t+2) 
\leq a(t-1)+a(t+2) \leq \cdots \leq a(1)+a(n-1)$.
If $a(2)>0$, then every second inequality in the inequality chains above is strict, implying that the minimum of $a(i)+a(n-i)$ is attained only at $i=t$ when $n=2t$, 
and only at $i=t$ or $i=t+1$ when $n=2t+1$.

Part (ii) is proved similarly.
\end{proof}

We call a sequence $\left(a(i)\right)_{i \ge 1}$ satisfying \eqref{EqSplit} 
for all $n\geq 2$ a {\em sequence of exponential type}, 
and we say that sequences 
$\left(a_1(i)\right)_{i \ge 1}$ and $\left(a_2(i)\right)_{i \ge 1}$
satisfying~\eqref{EqSplit3} for all $n\geq 2$ form a 
{\em pair of sequences of exponential type}.


The following lemma introduces a representation of integers that plays a central role in our study.

\begin{lem}\label{Lem_mkc}
Let $b\geq 2$ be a fixed integer.
\begin{enumerate}[(i)]
\item 
Every integer $u > b$ has a unique 
representation in the form
\beql{EqStd}
u ~=~ c (b^{\m}+1) + k\,,
\eeq
where ${\m} \ge 1$, $1 \le c \le b-1$, and 
\beql{Eqk}
\begin{cases}
0 \le k \le b^{\m},  &\text{if $c < b-1$}; \\
0 \le k \le b^{\m} - b + 1,   &\text{if $c=b-1$}. \\
\end{cases}
\eeq 

\item
Let $u=c(b^m+1)+k$ and $u'=c'(b^{m'}+1)+k'$ be positive integers 
represented as in~\eqref{EqStd}. 
Suppose that $u\leq u'$. Then $m\leq m'$. Furthermore, if $m=m'$, then $c\leq c'$. Finally, if $m=m'$ and $c=c'$, 
then $k\leq k'$.\footnote{In other words, if $u\leq u'$, the triple $(m,c,k)$ is lexicographically smaller than the triple $(m',c',k')$.}

\end{enumerate}
\end{lem}

\begin{proof} It is not hard to see that the representation 
\eqref{EqStd} uniquely defines $m=\Floor{ \log_b (u-1) }$,
$c= \Floor{ \tfrac{u}{b^m+1} }$, and $k=u - c (b^m+1)$.
The value of $m$ is a non-decreasing function of $u$, 
and so is $c$ when $m$ is fixed, and so is $k$ when $m$ and $c$ are fixed.
\end{proof}

The following examples illustrate how numbers are
represented in the form \eqref{EqStd}:
\begin{itemize}
\item for $b+1 \le u \le b^2$, we have $m=1$, $c = \big\lfloor \frac{u}{b+1} \big\rfloor$, and $k=u-(b+1)c$\,;
\item for $u=b^r$ with $r \ge 2$, we have $m = r-1$, $c=b-1$, and $k=b^{r-1}-b+1$. 
\end{itemize}

We will also need the following technical lemma.


\begin{lem}\label{lem_bound}
For any integers $b\geq 2$ and $m\geq 2$, 
$$(b-1)m - 2 \leq b^m - b + 1$$
and
$$b^m \geq \frac{1}{2}b^m + 2(b-1) \geq b^{m-1} + 2(b-1)\,.$$
\end{lem}

We omit the elementary proof.



\section{The recurrence for \texorpdfstring{$F(u)$}{F(u)}}\label{SecRec}

The following result, in particular the recurrence \eqref{EqRec}, is the key
to the whole paper.


\begin{thm}\label{ThRec}
We have $\Gen(u) = \emptyset$ and $F(u)=0$ when $u<0$ or $u=1$, 
and $\Gen(0) = \{0\}$, $F(0)=1$.
For $u \ge 2$, consider its representation 
$u=c (b^m+1)+k$ defined in Lemma~\ref{Lem_mkc}.
Then\footnote{We remark that the identity \eqref{EqRec} does not mention $c$ and 
holds even when the argument $(b-1){\m}-k-2$ is negative.}
\beql{EqGen}
\Gen(u) ~=~ \{cb^{\m}+v \mid v \in \Gen(k) \} 
~\cup~
\{cb^{\m} - 1 - v \mid v \in \Gen((b-1){\m}-k-2) \}
\eeq
and
\beql{EqRec}
F(u) ~=~ F(k) ~+~ F((b-1){\m}-k-2)\,.
\eeq
\end{thm}
\noindent

\begin{proof}
The first assertion is clear (there is no $v \in \NN$ such that
$f(v)=1$), and \eqref{EqRec} follows at once from \eqref{EqGen}.
To prove \eqref{EqGen}, let $u ~=~ c(b^{\m}+1) + k$ as in Lemma~\ref{Lem_mkc}.
We will show that (i) any element of the right-hand side of \eqref{EqGen}
is a generator of $u$, and (ii) every generator of $u$ is an element of the
right-hand side of \eqref{EqGen}.

(i) Suppose $v \in \Gen(k)$. 
Since $k \le b^{\m}$, $v < b^{\m}-1$ (by \eqref{Eq107_3}), 
we have $s(cb^{\m}+v) = c+s(v)$, and
$f(cb^{\m}+v)=cb^{\m}+v+c+s(v)=c(b^{\m}+1)+k=u$.
On the other hand, suppose $v \in \Gen((b-1){\m}-k-2)$. Let $w = cb^{\m}-1-v$.
By Lemma \ref{Lem3}, $s(w)=(b-1){\m}+c-1-s(v)$
(the condition $v \le cb^{\m}-2$ follows from $v < (b-1){\m}-k-2$).
Then $f(w) = w+s(w) = cb^{\m}-1-v+(b-1){\m}+c-1-s(v) = c(b^{\m}+1)+k = u$.

(ii) Suppose $w$ is a generator for $u = c(b^{\m}+1)+k$. 
Clearly, $u \le b^{{\m}+1}$, and $u=b^{{\m}+1}$ only 
when $c=b-1$ and $k=b^{\m}-b+1$. 
Trivially, either $w \ge cb^{\m}$ or $w<cb^{\m}$.

First, suppose $w \ge cb^{\m}$ and write it as $w=cb^{\m}+v$. 
If $v<b^{\m}$ then $s(w)=c+s(v)$, and $w+s(w)=u$ implies $v+s(v)=k$ and
$v \in \Gen(k)$. 
If $v \ge b^{\m}$ then 
if $c=b-1$, $w \ge b^{{\m}+1} \ge u$, contradicting
\eqref{Eq107_3}. 
So $c \le b-2$ and $w = (c+1)b^{\m}+\mu$, where $\mu=v-b^{\m}\geq 0$, 
which implies $u\geq (c+1)b^{\m}$, that is, $u = (c+1)b^{\m} + \lambda$, where $\lambda = c+k-b^{\m} \leq c$.
But $w+s(w)=u$ implies $\mu + s(\mu) + c+1 ~=~ \lambda$, a contradiction.
So $v \ge b^{\m}$ cannot happen.

Second, suppose $w< cb^{\m}$, that is, $w = cb^{\m}-1-v$ 
with $0 \le v \le cb^{\m}-2$. By Lemma \ref{Lem3}, $s(w)=(b-1){\m}+c-1-s(v)$,
and $w+s(w)=u$ implies $v+s(v)=(b-1){\m}-k-2$ and
$v \in \Gen((b-1){\m}-k-2)$.
\end{proof}

For example, in base $10$, if $u=10^{13}+1$, 
we have $c=1$, ${\m}=13$, $k=0$, so
$\F(10^{13}+1) = \F(0) + \F(115)$. Now $115 = 10^2+1+14$, 
so $\F(115)=\F(14)+\F(2) = 1+1$, and therefore $\F(10^{13}+1)=3$.
In Section~\ref{SecPropK} we will confirm Kaprekar's result that there
is no smaller number with three inverses.



Theorem~\ref{ThRec} can be used for computing the values of $\K(n)$ for small $n$.
In~\ref{AppPARI}, we provide a PARI/GP program that implements the formula \eqref{EqGen}.
We used it to compute the entries below $10^{10}$ in Tables~\ref{TabK1} and~\ref{TabK2}, which show the values of
$\K(n)$ for $n \le 7$ and bases $2 \le b \le 10$. 
The values of $\K(2)$ and $\K(3)$ for any base $b$ will be 
derived in the next section, and the values of $\K(n)$ for 
any $n$ and bases $2$ and $3$ in Section~\ref{SecB2}.
The values of $\K(n)$ in Tables~\ref{TabK1} and~\ref{TabK2}
for $n \ge 4$ and bases $b \ge 4$ are included here for convenience,
but they will not be officially established
until we have the recurrences of Section~\ref{SecBb}.\footnote{The values of $\K(2)$ and $\K(3)$ could also be
obtained from the recurrences of Section~\ref{SecBb}, 
but it seems more informative to calculate them directly.}

\begin{table}[thb]
$$
\begin{array}{|c|ccccc|}
\hline
b       &  2 & 3 & 4 & 5 & 6   \\
\hline
\K(1) & 0 & 0 & 0         & 0 & 0 \\
\K(2) & 2^2+1 & 3+1 & 4^2+1       & 5+1 & 6^2+1 \\
\K(3) & 2^7+1 & 3^3+1 & 4^7+1       & 5^2+1 & 6^9+1 \\
\K(4) & 2^{12}+6 & 3^5+5 & 4^{12}+18       & 5^4+7 & 6^{16}+38 \\
\K(5) & 2^{136}+6 & 3^{17}+5 & 4^{5468}+18       & 5^9+9 & 6^{(6^9+44)/5}+38  \\
\K(6) & 2^{260}+130 & 3^{29}+29 & 4^{10924}+4^7+2          & 5^{15}+27 & 6^{(2 \cdot 6^9 +8)/5}+6^9+2 \\
\K(7) & 2^{4233}+130 & 3^{139}+29 & 4^{\B_4(7)}+4^7+21,          & 5^{165}+27 & 6^{\B_6(7)}+6^9+2, \\
        &                         &                    & {}^{\B_4(7)=(4^{12}+4^7+40)/3} &                   & {}^{\B_6(7)=(6^{16} + 6^9 + 43)/5} \\
\hline
\end{array}
$$
\caption{Values of $\K(1),\ldots, \K(7)$ for bases $b=2,\ldots,6$ (the columns are 
\seqnum{A230303}\label{A230303b}, 
\seqnum{A230640}\label{A230640b}, 
\seqnum{A230638}\label{A230638b}, 
\seqnum{A230867}\label{A230867}, 
\seqnum{A238840}\label{A238840}).}
\label{TabK1}
\end{table}

\begin{table}[thb]
$$
\begin{array}{|c|cccc|}
\hline
b       &  7 & 8 & 9 & 10   \\
\hline
\K(1) & 0 &  0                                          & 0  & 0 \\
\K(2) & 7+1 &  8^2+1                                         & 9+1  & 10^2+1 \\
\K(3) & 7^2+1 & 8^{11}+1                  & 9^2+1  & 10^{13}+1 \\
\K(4) & 7^3+9 &  8^{20}+66              & 9^3+11  & 10^{24}+102 \\
\K(5) & 7^{10}+9 &  8^{\B_8(5)}+66,     & 9^{12}+11  & 10^{\B_{10}(5)}+102, \\
        &      &  {}^{\B_8(5)=(8^{11}+76)/7} &    & {}^{\B_{10}(5)=(10^{13}+116)/9}    \\
\K(6) & 7^{17}+51 &  8^{\B_8(6)}+8^{11}+2,                                           & 9^{21}+83  & 10^{\B_{10}(6)}+10^{13}+2, \\
        &                   &  {}^{\B_8(6)=(2 \cdot 8^{11}+12)/7} &                     & {}^{\B_{10}(6)=2  (10^{13}+8)/9}    \\
\K(7) & 7^{67}+51 &  8^{\B_8(7)}+8^{11}+2,                                           & 9^{103}+83  & 10^{\B_{10}(7)}+10^{13}+2, \\
        &                   &  {}^{\B_8(8)=(8^{20}+8^{11}+75)/7} &                     & {}^{\B_{10}(7)=(10^{24}+10^{13}+115)/9}    \\
\hline
\end{array}
$$
\caption{Values of $\K(1),\ldots, \K(7)$ for bases $b=7,\ldots,10$ (the columns are 
\seqnum{A238841}\label{A238841}, 
\seqnum{A238842}\label{A238842}, \seqnum{A238843}\label{A238843}, 
\seqnum{A006064}\label{A006064b}).}
\label{TabK2}
\end{table}


\section{Properties of \texorpdfstring{$\K(n)$}{K(n)}}\label{SecPropK}


When $u = \K(n)$, the smallest number with $n$
generators in base $b$, Theorem~\ref{ThRec} allows us
to make a stronger assertion than \eqref{EqStd}.
(Note that $\K(1)= 0$ for any base.)

\begin{thm}\label{Th131_1}
Let $b \ge 2$ and $n \ge 2$.
\begin{enumerate}[(i)]
\item $\K(n)$ has the following representation in the form~\eqref{EqStd}:
\beql{Eq131_1}
\K(n) ~=~ b^{\B(n)}+1+k\,,
\eeq
where the exponent $\B(n)$ is at least $1$ 
and\,\footnote{By Lemma~\ref{lem_bound}, the upper bound for $k$ here is better than~\eqref{Eqk} when $\B(n)\geq 2$, 
which holds for $n\geq 3$ (as we will show below).}
$0 \le k \le (b-1)\B(n)-2$.

\item If $b$ is odd, then both $\K(n)$ and $k$ are even.
\end{enumerate}
\end{thm}

\begin{proof}
Let $n\geq 2$. We notice that no integer from $1$ to $b$ can have more than one generator, and hence $\K(n)>b$.

(i) Let $u := \K(n)$, and write $u=c(b^{\m}+1)+k$ as in Lemma~\ref{Lem_mkc}.
If $c>1$, then we let $u':=b^k+1+k$ and notice that, by~\eqref{EqRec},
$\F(u')=\F(u)=n$ while $u'<u$, a contradiction. Hence $c=1$.

If $k > (b-1){\m}-2$, we apply~\eqref{EqRec} and obtain
$n = \F(u) = \F(k)+\F((b-1){\m}-k-2)$. Since the argument of the
last term is negative, we have $\F(u)=\F(k)$ while $k<u$, a contradiction.
Hence
$k\leq (b-1){\m}-2$. 

(ii) If $b$ is odd, $\K(n)$ is even by Lemma \ref{Lem1}, and 
therefore $k$ is even.
\end{proof}

In Theorems~\ref{ThK2} and \ref{ThK3} we compute 
the values of $\K(2)$ and $\K(3)$ and the corresponding values 
of $\B(2)$ and $\B(3)$.


\begin{thm}\label{ThK2}
For any $b\geq 2$, we have
\beql{EqK2}
\K(2) ~=~ 
\begin{cases}
b^2+1,  &\text{if $b$ is even}; \\
b + 1,   &\text{if $b$ is odd}; \\
\end{cases}
\qquad
\B(2) ~=~ 
\begin{cases}
2,  &\text{if $b$ is even}; \\
1,   &\text{if $b$ is odd}. \\
\end{cases}
\eeq
\end{thm}
\begin{proof}
Suppose $b$ is even. The number $b^2+1= [101]_b$ has the two generators
$b^2 = [100]_b$ and $b^2-b+1 = [b-1,1]_b$, so $\K(2) \le b^2+1$.
However, it is easy to check by hand that the values of
$f(v)$ for $0 \le v \le b^2$ are all distinct, so $\K(2)= b^2+1$.
The case when $b$ is odd is even easier to verify, and we omit the details.
\end{proof}


\begin{thm}\label{ThK3}
For any $b\geq 2$, we have
\beql{EqK3}
\K(3) ~=~ 
\begin{cases}
129,  &\text{if $b=2$}; \\
28,   &\text{if $b=3$}; \\
b^{b+3}+1,  &\text{if $b \ge 4$ is even}; \\
b^2 + 1,   &\text{if $b \ge 5$ is odd}; \\
\end{cases}
\qquad
\B(3) ~=~ 
\begin{cases}
7,  &\text{if $b=2$}; \\
3,   &\text{if $b=3$}; \\
b+3,  &\text{if $b \ge 4$ is even}; \\
2,   &\text{if $b \ge 5$ is odd}. \\
\end{cases}
\eeq
\end{thm}
\begin{proof}
For $b\in\{2,3,4\}$, we refer to Table \ref{TabK1}.
Suppose first that $b \ge 6$ is even.
Certainly $b^{b+3}+1$ has three generators, 
namely $b^{b+3}-1-[1,0,b-3]_b$,
$b^{b+3}-1-[b-1,b-2]_b$, and $b^{b+3}$ (this is easily checked
using Lemma~\ref{Lem3}).
So $\K(3) \le b^{b+3}+1$.
If $u:=\K(3)<b^{b+3}+1$, then by Theorem~\ref{Th131_1} we would have
$u = b^m+1+k$ with $m \le b+2$ and $0 \le k \le (b-1)m-2$.
By \eqref{EqRec} we have
$3 = F(u) = F(k)+F((b-1)m-k-2)$.
So either $F(k)=1$ and $F((b-1)m-k-2)=2$,
or $F(k)=2$ and $F((b-1)m-k-2)=1$.
We discuss only the first possibility, the second being similar.
From $k \ge \K(1) = 0$ and
$(b-1)m-k-2 \ge \K(2) = b^2+1$, we find that $m$
must equal $b+2$, and $0 \le k \le b-5$. 
By Lemma~\ref{Lem1}, $k$ is even. 
Let $\la := b-5-k$, which is odd, and $0 \le \la \le b-5$.
But now $F(b^2+\la+1) = 2 = F(\la)+F(\la-4)$ implies 
(again by Lemma~\ref{Lem1}) that $\la$ is even, a contradiction.
This completes the proof of $\K(3)=b^{b+3}+1$ for even $b \ge 6$.

Now, suppose that $b \ge 5$ is odd. 
Certainly $b^2+1$ has three generators,
$b^2 -1 - [1,(b-5)/2]_b$,
$b^2-1-[b-2]_b$, and $b^2$,
and we can easily check that at most
two of the values $f(v)$ for $0 \le v \le b^2$  can coincide.
\end{proof}

Theorem~\ref{ThK3} confirms Kaprekar's result that $\K(3)=10^{13}+1$ in base $10$.

When we come to study the case of a general base $b$ in Section~\ref{SecBb},
we will need to know the values of 
a refined version of $\K(n)$.
We define $\K_i(n)$ to be the smallest number $v \equiv i \pmod{b-1}$ for which $F(v)=n$,
where $i$ is a residue class modulo $b-1$.
If $b$ is odd, by Lemma~\ref{Lem1} we need only consider
even values of $i$, so we can say more precisely that $i \in \sI$,
where $\sI$ is a subset of the residue classes modulo $b-1$ given by
\beql{sIdef}
\sI := \begin{cases}
 \{0,1,2,3, \ldots, b-2\}, &  \text{if $b$ is even},\\ 
  \{0,2,4,6, \ldots, b-3\}, & \text{if $b$ is odd}.
\end{cases}
\eeq
Then we have: 
\beql{EqGTM}
\K(n) ~=~ \min_{i \in \sI} \K_i(n)\,.
\eeq

\noindent
There is an analog of Theorem~\ref{Th131_1} for $\K_i(n)$.


\begin{thm}\label{Th_form_Ki}
For any $b \ge 2$, $i\in\sI$, and $n \ge 2$, 
$\K_i(n)$ has the following representation in the form~\eqref{EqStd}:
\beql{Eq_form_Ki}
\K_i(n) ~=~ c (b^{\B(n)}+1)+k\,,
\eeq
where $\B(n)$ is as in \eqref{Eq131_1},\footnote{However, the $k$ in 
\eqref{Eq_form_Ki}
is not the same as the $k$ in \eqref{Eq131_1}.}
for some integers $c$ and $k$ satisfying $1\le c\le b-1$ and 
$0 \le k \le (b-1)\B(n)-2$.
Furthermore, if $b$ is odd then $c \le \frac{b-1}{2}$.
\end{thm}

\begin{proof}
Theorem~\ref{Th131_1} states that $\K(n)=b^{\B(n)}+1+k'$ for some $k'$, 
and so we have $\F(b^{\B(n)}+1+k')=n$. 
Theorem~\ref{ThRec} (applied twice) implies that for {\em any} $c'$ 
and the same $k'$, 
$$
\F(c'(b^{\B(n)}+1)+k')=\F(k')+\F((b-1)\B(n)-k'-2) =
\F(b^{\B(n)}+1+k')  = n\,.
$$
Let us show that there exists a value $c' := c_i$ such that 
$c_i(b^{\B(n)}+1)+k'\equiv i\pmod{b-1}$. 
Since $c_i(b^{\B(n)}+1)+k'\equiv k'+2c_i\pmod{b-1}$, 
we want $k'+2c_i\equiv i\pmod{b-1}$.
For even $b$, this congruence is trivially solvable for $c_i$ 
in the interval 
$1\le c_i\le b-1$. For odd $b$, $i\in\sI$ is even and so is 
$k'$ (by Theorem~\ref{Th131_1}), so the congruence reduces to $\frac{k'}2+c_i\equiv \frac{i}2\pmod{\frac{b-1}2}$, which is solvable for $c_i$ 
in the interval $1\le c_i\le \frac{b-1}{2}$.

By the definition of $\K_i(n)$, we have
\beql{Eq_Ki_int_bounds}
b^{\B(n)}+1+k'=\K(n)\le \K_i(n)\le c_i(b^{\B(n)}+1)+k'.
\eeq
By Theorem~\ref{ThRec}, $\K_i(n) = c(b^m+1)+k$ for some 
integers $c,k$, $1\leq c\leq b-1$, $0\leq k\leq b^m$.
From the inequality~\eqref{Eq_Ki_int_bounds} and Lemma~\ref{Lem_mkc}, 
we conclude that $m=\B(n)$ and 
$c\leq c_i$, and thus $c\leq \frac{b-1}2$ if $b$ is odd.
We further apply~\eqref{EqRec} to obtain
$n = F(c(b^{\B(n)}+1)+k) = F(k)+F((b-1)\B(n)-k-2)$.
If $k>(b-1)\B(n)-2$, then $F(k)=n$ and thus $k\geq\K(n)=b^{\B(n)}+1+k'$, 
which contradicts $k\leq b^{\B(n)}$.
Hence $k\leq (b-1)\B(n)-2$.
\end{proof}

For bases $b=2$ and $3$, we have $\sI = \{0\}$,
so there is only one $\K_i(n)$, which is $\K_0(n) = \K(n)$.
For bases $b \ge 4$ and $n \le 3$, we can give $\K_i(n)$ explicitly.
In the following three theorems, the subscripts $i$ in $\K_i(n)$
are elements of $\sI$, and in particular are to be read modulo $b-1$.
We omit the proofs, which are similar to those of 
Theorems~\ref{ThK2} and~\ref{ThK3}. 
Table~\ref{TabKiEG} illustrates these theorems.

\begin{thm}\label{ThKi1}
For even $b \ge 2$,
\begin{align}\label{EqWind1}
\K_{2 \lambda}(1) &= 2\lambda  \quad \text{for}\quad
0 \le \lambda \le \frac{b-2}{2}, \nonumber \\
\K_{2 \lambda+1}(1) &= b+2\lambda \quad \text{for}\quad
0 \le \lambda \le \frac{b-4}{2}\,;
\end{align}
for odd $b \ge 3$,
\beql{EqWind1b}
\K_{2\lambda}(1) = 2\lambda \quad\text{for}\quad
0 \le \lambda \le \frac{b-3}{2}\,.
\eeq
\end{thm}

\begin{thm}\label{ThKi2}
For even $b \ge 2$,
\beql{EqWindKi2e}
\K_{2 + 2 \lambda}(2) = b^2 + 1 + 2\lambda\quad\text{for}\quad 
0 \le \lambda \le b-2 \,;
\eeq
for odd $b \ge 3$,
\beql{EqWindKi2d}
\K_{2 + 2 \lambda}(2) = b + 1 + 2\lambda\quad\text{for}\quad
0\le \lambda \le \frac{b-3}{2}\,.
\eeq
\end{thm}

\begin{thm}\label{ThKi3}
For even $b \ge 4$,
\beql{EqWindKi3e}
\K_0(3) = b^{b+3}+b^2+2b-4\,,
\eeq
\beql{EqWindKi3e2}
\K_{2 + 2 \lambda}(3) = b^{b+3} + 1 + 2\lambda\quad\text{for}\quad
0 \le \lambda \le b-3\,;
\eeq
for odd $b \ge 5$,
\beql{EqWindKi3d}
\K_0(3) = b^2+2b-3\,,
\eeq
\beql{EqWindKi3d2}
\K_{2 + 2 \lambda}(3) = b^2 + 1 + 2\lambda\quad\text{for}\quad
0 \le \lambda \le \frac{b-5}{2}\,.
\eeq
\end{thm}

\begin{table}[htb]
$$
\begin{array}{|c|ccc||c|ccc|}
\hline
 \multicolumn{4}{|c||}{b=6} &  \multicolumn{4}{|c|}{b=9} \\
\hline
i \setminus n  & 1 & 2 & 3 & i \setminus n & 1 & 2 & 3 \\
\hline
0 & 0 & 45 & 6^9+44 & 0 & 0 & 16 & 96 \\
1 & 6 & 41 & 6^9+5  & 2 & 2 & 10 & 82 \\
2 & 2 & 37 & 6^9+1  & 4 & 4 & 12 & 84 \\
3 & 8 & 43 & 6^9+7  & 6 & 6 & 14 & 86 \\
4 & 4 & 39 & 6^9+3  &   &   &    &    \\
\hline
\end{array}
$$
\caption{Values of $\K_i(n)$ ($n \le 3$) for
bases $b=6$ (left) and $b=9$ (right), illustrating
Theorems \ref{ThKi1}-\ref{ThKi3}. }
\label{TabKiEG}
\end{table}

For $b \ge 4$, the minimal values of $\K_i(2)$ and $\K_i(3)$
occur when $\lambda = 0$, that is, when $i=2$, and confirm (via \eqref{EqGTM})
the values of $\K(2)$ and $\K(3)$ given in Theorems \ref{ThK2} and~\ref{ThK3}.

The following bounds on $\K_i(n)$ will be used in
the proof of Theorem \ref{ThBB}.


\begin{thm}\label{Th_Ki_bounds}
For any $b \ge 2$, $i\in\sI$, and $n \ge 2$, we have\,\footnote{From Theorem~\ref{Th131_1} it follows that $\K(n) \leq 2b^{\B(n)}$, which is a stronger upper bound for $\K(n)$ when $b\geq 4$.}
$$
b^{\B(n)} ~<~ \K(n) ~\le~ \K_i(n) ~<~ \beta \,b^{\B(n)+1}\,, 
$$
where $\beta := 1$ if $b$ is even, and $\beta := \frac{1}{2}$ if $b$ is odd.
\end{thm}

\begin{proof}
The lower bound for any $b$ and the upper bound for an even $b$ follow 
directly from Theorem~\ref{Th_form_Ki}. Let us prove that for odd $b$, 
$\K_i(n) \leq \frac{1}{2}b^{\B(n)+1}$.

For $n=2$, the bound can be verified directly using Table~\ref{TabK1} 
(for $b=3$) and Theorem~\ref{ThKi2} (for odd $b\geq 5$). Suppose $n\geq 3$.
By Theorem~\ref{Th_form_Ki}, we have $\K_i(n) = c(b^m+1)+k$, where $m=\B(n)$, 
$1\le c\le \frac{b-1}{2}$ and $0 \le k \le (b-1)m-2$. 
Furthermore, for $b=3$ we have $m\geq 3$ (see Table~\ref{TabK1}); 
while for odd $b\geq 5$, we have $m\geq 2$ by
Theorem~\ref{ThKi3}.
Hence
$$
\K_i(n) = c(b^m+1)+k \leq 
\frac{b-1}{2}(b^m+1) + (b-1)m-2 < \frac{b-1}{2}(b^m + 2m + 1)\,.
$$
It is easy to check that for $b=3$, $m\geq 3$ and $b\geq 5$, $m\geq 2$, 
we have $2m+1\leq \frac{1}{b-1}b^m$, implying that
$$
\K_i(n) < \frac{b-1}{2}\left(b^m + \frac{1}{b-1}b^m\right) =
\frac{1}{2}b^{m+1}
$$
as required.
\end{proof}



\section{\texorpdfstring{$\K(n)$}{K(n)} for bases \texorpdfstring{$2$}{2} 
and \texorpdfstring{$3$}{3}}\label{SecB2}

\begin{table}[htbp]
$$
\begin{array}{|r|rrrrrrrrrrrrrrrrrrr|}
\hline
u  & 0 & 1 & 2 & 3 & 4 & 5 & 6 & 7 & 8 & 9 & 10 & 11 & 12 & 13 & 14 & 15 & 16 & 17 & 18 \\
\hline
f(u) & 0 & 2 & 3 & 5 & 5 & 7 & 8 & 10 & 9 & 11 & 12 & 14 & 14 & 16 & 17 & 19 & 17 & 19 & 20 \\
F(u) & 1 & 0 & 1 & 1 & 0 & 2 & 0 & 1 & 1 & 1 & 1 & 1 & 1 & 0 & 2 & 0 & 1 & 2 & 0 \\
\hline
\end{array}
$$
\caption{Values of $f(u)$ and $F(u)$ in base $2$ 
(\seqnum{A092391}\label{A092391}, 
\seqnum{A228085}\label{A228085})}
\label{TabB2a}
\end{table}

We first discuss the base $2$ case for general $n$.
The initial values of $f(u)$ and $F(u)$ are shown in Table~\ref{TabB2a}.
We see that the smallest numbers with 1 and 2 
generators are $\K(1)=0$ and $\K(2)=5$, respectively.
Direct search by computer gives $\K(3)=129$ and $\K(4)=4102$,
as we have already seen in Table~\ref{TabK1}
(although $\K(5)= 2^{136}+6$ is out of reach).
The general solution is given by the following pair of recurrences.


\begin{thm}\label{ThBase2}
For $b=2$ and any $n \ge 2$, we have
\beql{EqB2}
\B(n) ~=~
\K\left(\Ceil{ \tfrac{n}{2} }\right)  +
 \K\left(\Floor{ \tfrac{n}{2} }\right) + 2
\eeq
and
\beql{EqKK2}
\K(n) ~=~ 2^{\B(n)} + 1 + \K\left(\Floor{ \tfrac{n}{2} }\right)\,.
\eeq
Also 
\beql{EqI2}
\K(n) ~>~ 2\K(n-1)\,.
\eeq
\end{thm}
\begin{proof}
The proof is by induction on $n$. The results are true for $n \le 3$,
so we assume $n \ge 4$. As in Theorem~\ref{Th131_1}, let
$u:=\K(n) = 2^{\m}+1+k$,
where $\m:=\B(n)$ and $0 \le k \le {\m}-2$.
By \eqref{EqRec}, 
$n = \F(u) = \F(k)+\F({\m}-k-2)$.
Let $\x:=\F(k)$, $\y:=\F({\m}-k-2)$ so that $\x + \y =n$.
Then $k \ge \K(\x)$, ${\m}-k-2 \ge \K(\y)$, and thus
\beql{Eq105}
{\m} \ge \K(\x)+\K(\y)+2\,.
\eeq
We know from \eqref{EqI2} that the sequence $\left(\K(n)\right)_{n\geq 1}$ is of exponential type,
so by Lemma~\ref{Lem4} the right-hand side
of \eqref{Eq105} is minimized only when either
$\x=\Ceil{ \tfrac{n}{2} }$, $\y=\Floor{ \tfrac{n}{2} }$
or
$\x=\Floor{ \tfrac{n}{2} }$, $\y=\Ceil{ \tfrac{n}{2} }$
(there is no difference if $n$ is even).
From Lemma~\ref{Lem_mkc}(ii), it follows that the value of $\B(n)$ is given by \eqref{EqB2}, and that $k$ equals the smaller of
$\K(\Ceil{ \tfrac{n}{2} })$ and $\K(\Floor{ \tfrac{n}{2} })$, which is $\K(\Floor{ \tfrac{n}{2} })$ by induction and \eqref{EqI2}. This proves \eqref{EqKK2}.
The proof of \eqref{EqI2} is now a routine calculation; we omit the details.
\end{proof}

\noindent \textbf{Remark.} The proof of Theorem~\ref{ThBase2} also shows that
\begin{align}\label{EqGen2}
\Gen&(\K(n)) ~=~ \nonumber \\
& \{ 2^{\B(n)} + v \mid v \in 
 \Gen \left( \K \left(\Floor{ \tfrac{n}{2} } \right)\right) \}
~\cup~
\{ 2^{\B(n)} - 1 - v \mid v \in 
\Gen \left(\K \left(\Ceil{ \tfrac{n}{2} } \right) \right) \}\,.
\end{align}

Table \ref{TabK2a} extends the $b=2$ column of Table \ref{TabK1}
to $n=16$. (The first 100 terms of $\B(n)$ and $\K(n)$
are given in the entries  \seqnum{A230302}\label{A230302} 
and \seqnum{A230303}\label{A230303} in \cite{OEIS}).

\begin{table}[htb]
$$
\begin{array}{|c|c|c|}
\hline
n       &  \B(n) & \K(n)    \\
\hline
8 & 8206 & 2^{8206}+4103  \\
9 & 2^{136}+4110 & 2^{\B(9)}+4103  \\
10 & 2^{137}+14 & 2^{\B(10)}+2^{136}+7  \\
11 & 2^{260}+2^{136}+138 & 2^{\B(11)}+2^{136}+7  \\
12 & 2^{261}+262 & 2^{\B(12)}+2^{260}+131    \\
13 & 2^{4233}+2^{260}+262 & 2^{\B(13)}+2^{260}+131    \\
14 & 2^{4234}+262 & 2^{\B(14)}+2^{4233}+131    \\
15 & 2^{8206}+2^{4233}+4235 & 2^{\B(15)}+2^{4233}+131    \\
16 & 2^{8207}+8208 & 2^{\B(16)}+2^{8206}+4104    \\
\hline
\end{array}
$$
\caption{ Base $2$: $\B(n)$ and $\K(n)$ for $n=8, \ldots, 16$,
extending Table \ref{TabK1}.}
\label{TabK2a}
\end{table}


There is a similar pair of recurrences in the base $3$ case.


\begin{thm}\label{ThBase3}
For $b=3$ and any $n \ge 2$, we have
\beql{EqB3}
\B(n) ~=~  \frac{ 
\K\left(\Ceil{ \tfrac{n}{2} }\right) +
\K\left(\Floor{ \tfrac{n}{2} }\right) +2 }{2}
\eeq
and
\beql{EqKK3}
\K(n) ~=~ 3^{\B(n)} + 1 + \K\left(\Floor{ \tfrac{n}{2} }\right)\,.
\eeq
Also 
\beql{EqI3}
\K(n) ~>~ 3\K(n-1)\,.
\eeq
\end{thm}
\begin{proof}
The proof is similar to that of Theorem \ref{ThBase2},
except at one step.
Again we use induction on $n \ge 4$ and let $u := \K(n) = 3^{\m}+1+k$,
where $m:=\B(n)$ and $0 \le k \le 2{\m}-2$.
Then
$n = F(u) = F(k)+F(2{\m}-k-2) = \x+\y$, say,
with  $\x + \y =n$.
Then $k \ge \K(\x)$, $2{\m}-k-2 \ge \K(\y)$, so
\beql{Eq1053}
2{\m} \ge \K(\x)+\K(\y)+2\,.
\eeq
The difference from \eqref{Eq105} in the base $2$ case lies in the presence
of the factor of $2$ (in general it will be $b-1$) on the left-hand side
of  this inequality.
So now we must minimize the sum $\K(x)+\K(y)$ subject to the 
additional requirement that the sum is even. Here that does
not cause any difficulty, because all values of $\K$ are even 
(by Lemma \ref{Lem1}).
We complete the proof as in the base $2$ case,
by taking $\x=\Floor{ \tfrac{n}{2} }$, $\y=\Ceil{ \tfrac{n}{2} }$.
\end{proof}

The first seven terms of $\B(n)$ and $\K(n)$ for base $3$ are shown in 
Table \ref{TabK1}; the first
100 terms may be found in \seqnum{A230639}\label{A230639} 
and \seqnum{A230640}\label{A230640}.



\section{\texorpdfstring{$\K(n)$}{K(n)} for a general
base \texorpdfstring{$b$}{b}}\label{SecBb}

In this section we give a set of recurrences
that determine $\K(n)$ for
a general base $b \ge 2$.
The divisibility requirement that we
encountered in \eqref{Eq1053} for the base $3$ case makes the recurrences
in the general case considerably more complicated.

We know from Theorem \ref{Th131_1} that $\K(n)$ has
the form 
\beql{Eq9.1}
\K(n) = b^{\B(n)}+1+k\,,
\eeq
where 
$0 \le k \le (b-1)\B(n)-2$.
Then by \eqref{EqRec},
$$
n = \F(\K(n)) = F(k)+F((b-1)\B(n) -k-2) = \x+\y\,,
$$
where $\x:=\F(k)$ and $\y:=\F((b-1)\B(n) -k-2)$.
Since both $k$ and $(b-1)\B(n) -k-2$ are smaller than $\K(n)$ and 
thus cannot have $n$ generators, 
the values of $\x,\y$ must be in the range from $1$ to $n-1$.

The definitions of $\x,\y$ imply
$k \ge \K(\x)$, $(b-1)\B(n)-k-2 \ge \K(\y)$, and therefore
\beql{Eqbstar2}
(b-1)\B(n)  \ge \K(\x)+\K(\y) + 2\,.
\eeq
Since in general $\K(\x)+\K(\y)+2$ will not be a multiple of $b-1$,
the implied lower bound on $\B(b)$ cannot always be attained. We therefore 
refine the inequality \eqref{Eqbstar2} using the functions $\K_i(n)$ 
introduced in Section \ref{SecPropK}, and replace \eqref{Eqbstar2} with an
inequality where the implied lower bound on $\B(n)$ {\em can} be attained. 
If $k\equiv i \pmod{b-1}$ for $i\in\sI$, then 
$(b-1)\B(n)-k-2\equiv -i-2 \pmod{b-1}$ and so
$k\geq \K_i(x)$, $(b-1)\B(n)-k-2\geq \K_{-i-2}(y)$, and
\beql{Eqbstar}
(b-1)\B(n)  \ge \K_i(\x)+\K_{-i-2}(\y) + 2\,.
\eeq
Now, in contrast to \eqref{Eqbstar2}, the right-hand side is
divisible by $b-1$, and so we obtain an 
integer-valued lower bound on $\B(n)$ (for some $i$ and $\x+\y=n$). 
Namely, \eqref{Eqbstar} implies
\beql{EqA11a}
\B(n) \ge \frac{\min_{i \in \sI} \min_{1\le\x\le n-1} ~ 
\K_i(\x)+\K_{-i-2}(n-\x) + 2}{b-1}\,.
\eeq

We will show by induction that for any $i \in \sI$, the sequences
$\left(\K_i(n)\right)_{n \ge 1}$ and $\left(\K_{-i-2}(n)\right)_{n \ge 1}$
form a pair of sequences of exponential type. 
Then by Lemma~\ref{Lem4}, we will be able to replace the inner 
minimum in~\eqref{EqA11a} with 
\beql{EqA11c}
\K'_i(n) ~:=~
 \min  
\left\{ 
\K_i\left(\Ceil{ \tfrac{n}{2} }\right)
+\K_{-i-2}\left(\Floor{ \tfrac{n}{2} }\right) 
,~ 
\K_i\left(\Floor{ \tfrac{n}{2} }\right)
+\K_{-i-2}\left(\Ceil{ \tfrac{n}{2} }\right)
 \right\}\,.
 \eeq
In fact, we will prove that equality holds in \eqref{EqA11a}, 
i.e., $\B(n)=\hat\B(n)$, where
 \beql{hatB}
\hat\B(n) ~:=~ \frac{\min_{i \in \sI} \K'_i(n) +2}{b-1}
~=~ 
\frac{\min_{i \in \sI} \K_i\left(\Ceil{ \tfrac{n}{2} }\right)
+\K_{-i-2}\left(\Floor{ \tfrac{n}{2} }\right)   +2}{b-1}\,,
 \eeq
where the latter expression follows from the symmetry between $i$ and $-i-2$.

For $n \ge 2$ and $i \in \sI$, we define 
\beql{Eqci}
c_{i,n} ~:=~ 
\text{~smallest~integer~} c\geq 1 \text{~such~that~} 
\K'_{i-2c}(n)=\min_{j\in\sI} \K'_j(n)
\eeq
and
\begin{align}\label{Eqhi}
 h_{i,n} :=  
 \Ceil{ \frac{n}{2} } \quad  \text{if}\  
 \K_{i-2c_{i,n}}\left(\Ceil{ \tfrac{n}{2} }\right)+& \K_{2c_{i,n}-i-2}\left(\Floor{ \tfrac{n}{2} }\right)
   \nonumber \\
 <& \K_{i-2c_{i,n}}\left(\Floor{ \tfrac{n}{2} }\right)+
  \K_{2c_{i,n}-i-2}\left(\Ceil{ \tfrac{n}{2} }\right);  \nonumber \\
 h_{i,n} :=
 \Floor{ \frac{n}{2} } \quad  \text{otherwise}\,.
\end{align}

These definitions allow us to express $\hat\B(n)$ as
\beql{hatBprop}
\hat\B(n) = \frac{\K'_{i-2c_{i,n}}(n)+2}{b-1} = 
\frac{\K_{i-2c_{i,n}}(h_{i,n})+\K_{2c_{i,n}-i-2}(n-h_{i,n})+2}{b-1}\,,
\eeq
which holds for any $i \in \sI$.


\begin{thm}\label{ThBB}
For all $n \ge 2$, 
\beql{EqS5}
\min_{1\leq j\leq n-1} \K_i(j) + \K_{-i-2}(n-j) 
= \K'_i(n)\ \text{~for all}\ i \in \sI\,,
\eeq
and
\beql{EqS1}
\B(n) ~=~ \hat\B(n)\,.
\eeq
Furthermore, for all $n \ge 3$,\footnote{Note that $\B(1)$ is not defined, which is why we start~\eqref{EqS4} at $n=3$.} 
\beql{EqS4}
\B(n) ~\ge~ 
\begin{cases} 
\B(n-1)+1, & \text{if $b$ is odd and $n\in\{3,4\}$}\,;\\
\B(n-1)+2, & \text{otherwise}\,.
\end{cases}
\eeq
\end{thm}
\begin{proof}
We will prove all three statements 
\eqref{EqS5}, \eqref{EqS1}, and \eqref{EqS4} 
together by induction on $n$.
We write~$\eqref{EqS5}_j$, $\eqref{EqS1}_j$,
$\eqref{EqS4}_j$, to refer to the statements~\eqref{EqS5}, \eqref{EqS1},
\eqref{EqS4} for $n=j$.
We divide the proof into the following four parts:  
\begin{enumerate}[(I)]
\item  $\eqref{EqS5}_2$, $\eqref{EqS1}_2$, $\eqref{EqS1}_3$, $\eqref{EqS4}_3$ are true.
\item For $n\geq 3$, $\eqref{EqS5}_n$ follows from $\eqref{EqS4}_{j}$ for $3\leq j<n$.
\item For $n\geq 4$, $\eqref{EqS1}_n$ follows from $\eqref{EqS5}_n$ and $\eqref{EqS4}_j$ for $3\leq j< n$.  
\item For $n\geq 4$, $\eqref{EqS4}_n$ follows from $\eqref{EqS1}_{n}$, $\eqref{EqS1}_{n-1}$, $\eqref{EqS4}_j$ for $3 \le j< n$.
\end{enumerate}

\noindent \textbf{\it Proof of (I).} 
This is easily verified using the
values of $\K(n)$ and $\K_i(n)$ from Theorems~\ref{ThK2}, \ref{ThK3}, \ref{ThKi1}, \ref{ThKi2}, 
and noticing that $c_{i,2}=1$ and $h_{i,2}=1$ for all $i\in\sI$.

\vspace*{0.1in}
\noindent \textbf{\it Proof of (II).} Let $n\geq 3$ and $i\in\sI$.
To establish $\eqref{EqS5}_n$, we will first show that
\beql{EqKsplit}
\K_i(m+1)  ~>~ \K_i(m) + \K_{-i-2}(m)\,
\eeq
holds for all $m=1,2,\dots,n-2$, 
and then apply Lemma~\ref{Lem4}(ii). 
The inequality~\eqref{EqKsplit}  can be verified directly for $m=1$ using Theorems~\ref{ThKi1} and~\ref{ThKi2}.
If $m \ge 2$, 
we consider two cases depending on the parity of $b$, and use 
Theorem~\ref{Th_Ki_bounds} to bound the terms in~\eqref{EqKsplit}.
For even $b$, 
we have $\K_i(m+1) > b^{\B(m+1)} \geq b^{\B(m)+2}$ 
by~$\eqref{EqS4}_{m+1}$,  
while the right-hand side of \eqref{EqKsplit} 
is at most $2b^{\B(m)+1}\leq b^{\B(m)+2}$ (since $b\geq 2$).
For odd $b$, we have 
$\K_i(m+1) > b^{\B(m+1)} \geq b^{\B(m)+1}$ by $\eqref{EqS4}_{m+1}$,  
while the right-hand side of \eqref{EqKsplit} is at most 
$b^{\B(m)+1}$.
This proves~\eqref{EqKsplit}, which by Lemma~\ref{Lem4}(ii) 
(taking $a(m)=\K_i(m)$, $b(m)=\K_{-i-2}(m)$) implies $\eqref{EqS5}_n$.

\vspace*{0.1in}
\noindent \textbf{\it Proof of (III).} Let $n\geq 4$. We fix an arbitrary $i\in\sI$.
To prove $\eqref{EqS1}_n$, 
we first use Theorem~\ref{Th_form_Ki} to write 
$\K_i(n) = c(b^{\B(n)}+1) + k$ 
for some integers $c$ and $k$ satisfying $1\le c\le b-1$ and 
$0 \le k \le (b-1)\B(n)-2$. Then \eqref{EqRec} implies
$n = \F(\K_i(n)) = \F(k) + \F((b-1)\B(n)-k-2)$.

Since $\K_i(n)\equiv i\pmod{b-1}$, we have $k\equiv i-2c\pmod{b-1}$. 
Let $x:=\F(k)$. 
Then $k\geq \K_{i-2c}(x)$ and 
$(b-1)\B(n)-k-2\geq \K_{2c-i-2}(n-x)$,
thus
\beql{EqBBprime}
\B(n)
~ \stackrel{\text{(i)}}{\geq}~
\frac{\K_{i-2c}(x) + \K_{2c-i-2}(n-x) + 2}{b-1}
~ \stackrel{\text{(ii)}}{\geq}~ 
\frac{\K'_{i-2c}(n)+2}{b-1}
~ \stackrel{\text{(iii)}}{\geq}~
\hat\B(n)\,,
\eeq
where inequality (i) is immediate,
(ii) follows from $\eqref{EqS5}_n$, and (iii) follows from \eqref{hatB}.

Conversely, we now prove that $\B(n) \le \hat\B(n)$. First, we notice 
that $\eqref{EqS4}_j$ for $3\leq j\leq \Ceil{\tfrac{n}{2}}$ and 
Theorem~\ref{ThK2} imply that $\B\left(\Ceil{\tfrac{n}{2}}\right)\geq \B(2) \geq 1$.
Together with~\eqref{hatBprop} and Theorem~\ref{Th_Ki_bounds}, this further gives 
\beql{eq_hatB2}
\hat\B(n) \geq \frac{\K\left(\Ceil{\tfrac{n}{2}}\right)+2}{b-1} >
\frac{b^{\B\left(\Ceil{\tfrac{n}{2}}\right)}+2}{b-1} \geq \frac{b+2}{b-1} > 1\,.
\eeq 
Now, let us define  
$L := c_{i,n} (b^{\hat\B(n)}+1) + 
\K_{i-2c_{i,n}}(h_{i,n})$, which has the form~\eqref{EqStd} since
$$
\K_{i-2c_{i,n}}(h_{i,n}) \leq (b-1)\hat\B(n) - 2 \leq b^{\hat\B(n)}-b+1
$$
as follows from~\eqref{hatBprop}, \eqref{eq_hatB2}, 
and Lemma~\ref{lem_bound}.
Using \eqref{EqRec} and \eqref{hatBprop},
we have
\beql{EqQQ1}
\begin{split}
F(L) &= F(\K_{i-2c_{i,n}}(h_{i,n})) + F((b-1)\hat\B(n)-\K_{i-2c_{i,n}}(h_{i,n})-2)\\
&= F(\K_{i-2c_{i,n}}(h_{i,n})) + F(\K_{2c_{i,n}-i-2}(n-h_{i,n}))\\
&= h_{i,n} + n-h_{i,n}\\ 
&= n\,.
\end{split}
\eeq
From~\eqref{EqQQ1} and $L \equiv 2c_{i,n} +(i-2c_{i,n}) \equiv i \pmod{b-1}$,
it follows that $\K_i(n)\leq L$. Then Lemma~\ref{Lem_mkc} 
implies that $\B(n)\leq \hat\B(n)$, which together 
with \eqref{EqBBprime} establishes $\eqref{EqS1}_n$.

\vspace*{0.1in}
\noindent \textbf{\it Proof of (IV).} 
For $n=4$, we use the identity $\eqref{EqS1}_4$ and Theorem~\ref{ThKi2} to obtain
\beql{EqE4}
\B(4) = 
\begin{cases} 
12, & \text{if $b=2$}\,; \\
5, & \text{if $b=3$}\,; \\
4, & \text{if $b=5$}\,; \\
2b+4, & \text{if $b\geq 4$ is even}\,; \\
3, & \text{if $b\geq 7$ is odd}\,.
\end{cases}
\eeq
Comparing these values to those of $\B(3)$ given in Theorem~\ref{ThK3}, we conclude that $\eqref{EqS4}_4$ holds.

Let $n\geq 5$.
To prove $\eqref{EqS4}_n$, we consider two cases depending on 
the parity of $n$.

First, suppose that $n$ is even, i.e., $n=2t$ for some $t\geq 3$. 
The identity~$\eqref{EqS1}_n$ gives
\beql{Eq324}
\B(n)  ~=~ 
\frac{\K_{i-2c_{i,n}}(t)+\K_{2c_{i,n}-i-2}(t)+2}{b-1}
 ~\ge~ \frac{2 \K(t)+2}{b-1}\,,
\eeq
while $\eqref{EqS1}_{n-1}$ gives
\beql{Eq325}
\B(n-1) ~=~
\frac{ 
\min_{j} \K_j(t)+\K_{-j-2}(t-1) 
+ 2}{b-1}\,. 
\eeq
We obtain an upper bound on the right-hand side of~\eqref{Eq325} 
if we choose any particular value of $j$,
so let us choose $j=\ell$ such that $\K_\ell(t) = \K(t)$. Then
$$
\B(n-1) ~\leq~ \frac{ \K(t) + \K_{-\ell-2}(t-1) + 2}{b-1}\,. 
$$
The inequality $\eqref{EqS4}_n$ will follow if we show that
$$
\frac{2 \K(t)+2}{b-1} \geq \frac{\K(t)+\K_{-\ell-2}(t-1)+2}{b-1}+2\,,
$$
or, equivalently,
\beql{EqKtKl}
\K(t) \geq \K_{-\ell-2}(t-1)+2(b-1)\,.
\eeq
In fact, \eqref{EqKtKl} holds for any $t$ such that 
$\eqref{EqS4}_t$ holds and any $\ell\in\sI$. To prove it, 
we consider two cases depending on the parity of $b$ and use Lemma~\ref{lem_bound},
Theorem~\ref{Th_Ki_bounds}, and $\eqref{EqS4}_t$. For even $b$, 
the inequality~\eqref{EqKtKl} follows from
$$
\K(t) > b^{\B(t)} \geq b^{\B(t-1)+2} \geq b^{\B(t-1)+1} + 2(b-1) >
\K_{-\ell-2}(t-1)+2(b-1)\,.
$$
For odd $b$, the inequality~\eqref{EqKtKl} follows from
$$\K(t) > b^{\B(t)} \geq b^{\B(t-1)+1} \geq \frac{1}{2}b^{\B(t-1)+1}
+ 2(b-1) > \K_{-\ell-2}(t-1)+2(b-1)\,.$$
This completes the proof of~\eqref{EqKtKl} and thus 
$\eqref{EqS4}_n$ for even $n$.

Second, suppose that $n$ is odd, i.e., $n=2t-1$ for some $t\geq 3$. Arguing as before, we have
$$
\B(n) \geq \frac{\K(t)+\K(t-1)+2}{b-1}
~\text{and}~\B(n-1)\leq 
\frac{\K(t-1)+\K_{-\ell-2}(t-1)+2}{b-1}
$$
for some $\ell\in\sI$.
We know from the previous paragraph that the
inequality~\eqref{EqKtKl} holds, and this again implies
$\eqref{EqS4}_n$.

This completes the proof of the theorem.
\end{proof}

From~\eqref{EqS5} we immediately have
\begin{cor}\label{CorKexp}
For any $i \in \sI$, the sequences 
$\left(\K_i(n)\right)_{n \ge 1}$ and $\left(\K_{-i-2}(n)\right)_{n \ge 1}$
form a pair of sequences of exponential type. 
\end{cor}


\begin{thm}\label{ThS2}
For any $b\geq 2$, $i \in \sI$, and $n\geq 2$, we have
\beql{EqS2}
\K_i(n) ~=~ c_{i,n} (b^{\B(n)}+1) + \K_{i-2c_{i,n}}(h_{i,n})\,.
\eeq
Moreover, the representation~\eqref{EqS2} has the form~\eqref{Eq_form_Ki}.
\end{thm}

\begin{proof} For $n=2$ and $n=3$, the statement follows from Theorems~\ref{ThK2}, \ref{ThK3}, \ref{ThKi1}, \ref{ThKi2}, and \ref{ThKi3}.
It can be verified that $c_{i,2}=c_{i,3}=1$ and $h_{i,2}=h_{i,3}=1$ for all $i\in\sI$, except for $h_{0,3}=2$ when $b\geq 4$.

For $n\geq 4$, the proof emerges from the proof of Part (III) of Theorem~\ref{ThBB}.
We use Theorem~\ref{Th_form_Ki} to write $\K_i(n) = c(b^{\B(n)}+1) + k$
for some integers $c$ and $k$ satisfying 
$1\le c\le b-1$ and $0 \le k \le (b-1)\B(n)-2$.
Since $\B(n)=\hat\B(n)$ (by Theorem~\ref{ThBB}), the chain of inequalities (i), (ii), 
and (iii) in \eqref{EqBBprime} are all equalities.

By~\eqref{hatB}, the equality (iii) is equivalent 
to $\K'_{i-2c}(n)=\min_{i \in \sI} \K'_i(n)$,
implying that $c\geq c_{i,n}$ (by the definition of $c_{i,n}$).
Again, we consider $L := c_{i,n}(b^{\hat\B(n)}+1) + \K_{i-2c_{i,n}}(h_{i,n})$, 
for which we proved that $\K_i(n)\leq L$. 
Since $\B(n)=\hat\B(n)$, Lemma~\ref{Lem_mkc} now implies 
$c\leq c_{i,n}$, and thus $c=c_{i,n}$.

Since $c=c_{i,n}$, the equality (i) implies $k=\K_{i-2c_{i,n}}(x)$, 
while the equality (ii) implies that
$\K_{i-2c_{i,n}}(x) + \K_{2c_{i,n}-i-2}(n-x)$ equals its minimum 
value $\K'_{i-2c_{i,n}}(n)$.
Then from Lemma~\ref{Lem4} it follows that $x=\Floor{\tfrac{n}{2}}$ 
or $x=\Ceil{\tfrac{n}{2}}$. 
Furthermore, $\K_i(n)\leq L$ implies that 
$k=\K_{i-2c_{i,n}}(x) \leq \K_{i-2c_{i,n}}(h_{i,n})$ 
(by Lemma~\ref{Lem_mkc}) and thus $x\leq h_{i,n}$. 
When $h_{i,n}=\Floor{\tfrac{n}{2}}$ (in particular, when $n$ is even), 
we must have $x=h_{i,n}$. 
On the other hand, if $n$ is odd and $h_{i,n}=\Ceil{\tfrac{n}{2}}$, 
then $x=\Floor{\tfrac{n}{2}}\ne h_{i,n}$ does not produce 
the minimum of $\K_{i-2c_{i,n}}(x) + \K_{2c_{i,n}-i-2}(n-x)$ as 
follows from~\eqref{Eqhi}. Hence in all cases we have $x=h_{i,n}$ 
and thus $k=\K_{i-2c_{i,n}}(h_{i,n})$.
in which cases $x=h_{i,n}$ by~\eqref{Eqhi}. In the case when $n$ is odd and the two sums above are equals, both $x=\Floor{\tfrac{n}{2}}=h_{i,n}$ and $x=\Ceil{\tfrac{n}{2}}=n-h_{i,n}$ deliver the minimum of $\K_{i-2c_{i,n}}(x) + \K_{2c_{i,n}-i-2}(n-x)$. However, here $\K_i(n)\leq L$ implies that $k=\K_{i-2c_{i,n}}(x) \leq \K_{i-2c_{i,n}}(h_{i,n})$ (by Lemma~\ref{Lem_mkc}) and thus $x=h_{i,n}$. Therefore, in all cases we have $k=\K_{i-2c_{i,n}}(h_{i,n})$, 
which completes the proof.
\end{proof}


\begin{cor}\label{Cor20}
\beql{EqCor20}
\K(n) ~\ge~ b^{\B(n)} + 1 + \K\left(\Floor{ \tfrac{n}{2} }\right)\,.
\eeq
\end{cor}

\begin{proof} From \eqref{EqGTM} and Theorem~\ref{ThS2} it follows that
$\K(n) = \K_i(n) = c_{i,n}(b^{\B(n)}+1)+K_{i-2c_{i,n}}(h_{i,n})$ for some $i \in \sI$, 
for which we also have $c_{i,n} = 1$ by Theorem~\ref{Th131_1} and $h_{i,n} \geq \Floor{ \tfrac{n}{2} }$ by \eqref{Eqhi}.
Hence 
$$
\K(n) = b^{\B(n)} + 1 + \K_{i-2}(h_{i,n})
\ge b^{\B(n)} + 1 + \K\left(\Floor{ \tfrac{n}{2} }\right)\,.
$$
\end{proof}

\noindent \textbf{Remarks on Theorems \ref{ThBB} and \ref{ThS2}.}
\begin{enumerate}[(1)]

\item 
Since, by definition, $\K_i(n)\equiv i\pmod{b-1}$, we have
$\K_i(n) \ne \K_j(n)$ for $i \ne j$ from $\sI$.
It follows that the choice of $i \in \sI$ in~\eqref{EqGTM} is unique,
and so we may define a ``generalized
Thue--Morse sequence" $\left(\tau(n)\right)_{n \ge 1}$ for base $b$ by:
\beql{EqTAU}
\K(n) = \begin{cases}
 \K_{\tau(n)}(n), &  \text{if $b$ is even}\,,\\
 \K_{2\tau(n)}(n), &  \text{if $b$ is odd}\,,
\end{cases}
\eeq
where $0 \le \tau(n) \le b-2$ if $b$ is even, and
$0 \le \tau(n) \le \frac{b-3}{2}$ if $b$ is odd.
The name comes from the fact that in base $b=5$ this is
the classical Thue--Morse sequence (see Section~\ref{SecB5}).

From \eqref{EqTAU} and Theorems~\ref{Th131_1} and \ref{ThS2} it follows that
\beql{EqFP2}
c_{\tau(n),n} = 1 \mbox{~if~ $b$~is~even}, \quad
c_{2\tau(n),n} = 1 \mbox{~if~ $b$~is~odd}\,.
\eeq

\item 
 The $\K_i'(n)$ are not all distinct. It follows at once from \eqref{EqA11c} 
 that:
\begin{enumerate}[(i)]
\item if $b$ is even, the distinct $\K'_i(n)$ are
\beql{EqKp1}
\K'_i(n), ~\text{for~}0 \le i \le \frac{b-4}2, \text{~and~} \K'_{b-2}(n)\,,
\eeq
where the remaining values are given by  $\K'_i(n) = \K'_{b-i-3}(n)$;

\item if $b \equiv -1 \pmod{4}$, the distinct $\K'_i(n)$ are
\beql{EqKp2}
\K'_{2i}(n), ~\text{for~}0 \le i \le \frac{b-7}4, \text{~and~} \K'_{\frac{b-3}2}(n)\,,
\eeq
where the remaining values are given by  $\K'_{2i}(n) = \K'_{b-2i-3}(n)$; 
and

\item if $b \equiv 1 \pmod{4}$, the distinct $\K'_i(n)$ are
\beql{EqKp3}
\K'_{2i}(n), ~\text{for~}0 \le i \le \frac{b-5}4\,,
\eeq
where the remaining values are given by  $\K'_{2i}(n) = \K'_{b-2i-3}(n)$.
\end{enumerate}

In fact, the calculations for the case $b=4m-1$ are 
(apart from a relabeling of the variables)
essentially the same as the calculations for the case $b=2m$.
For $m=1$, this can be seen from Theorems~\ref{ThK2} and \ref{ThK3}, which establish similar recurrences for $\K(n)$ in bases $b=2$ and $b=3$.
We will formally prove this similarity for general $m$ in Theorem~\ref{Thm_b1b2} in the next section.

\item 
If $n$ is even, no minimization is needed in the formula \eqref{EqA11c} for $\K'_i(n)$,
since the two terms inside the braces are the same, and also 
$h_{i,n} = \tfrac{n}{2}$ for all $i \in \sI$.
If $n$ is odd, the two terms inside the braces in~\eqref{EqA11c} may
still coincide, including the case of $\K'_{i-2c_{i,n}}(n)$ (which happens always when $b=2,3$, quite frequently when 
$b\geq 4$ is even, and sometimes when $b\geq 5$ is odd).
While both $\Floor{ \tfrac{n}2 }$ and $\Ceil{ \tfrac{n}2 }$ in this case 
may serve as $h_{i,n}$ in~\eqref{hatBprop}, the choice of 
$h_{i,n}=\Floor{ \tfrac{n}2 }$ is dictated by Theorem~\ref{ThS2}, 
which expects the contribution of $\K_{i-2c_{i,n}}(h_{i,n})$ be as small 
as possible.

\item We initially thought that the minimization in~\eqref{hatB}
would be determined by choosing the index $i$ to be either 
$j$ or $-j-2$, where
$\K_j\left(\Ceil{ \tfrac{n}{2} }\right) = \K\left(\Ceil{ \tfrac{n}{2} }\right)$.
This would imply that
\beql{EqRed}
\min_{i \in \sI} \K_i'(n) ~ \stackrel{?}{=} ~ 
\K\left(\Ceil{ \tfrac{n}{2} }\right)+
\K_{\ell}\left( \Floor{ \tfrac{n}{2} } \right)\,,
\eeq
for some $\ell \in \sI$.
To prevent others from falling into this trap, we mention 
that~\eqref{EqRed} is false. There are counter-examples 
when $b=7$ and $n=13$, and when $b=9$ and $n=9$ 
(see details in~\ref{AppFlow}).
\end{enumerate}


\paragraph{Computing $\B(n)$ and $\K(n)$.} It may be helpful to summarize
the steps involved in using the recurrences to compute $\B(n)$ and $\K(n)$:
\begin{enumerate}[Step 1:]
\item For every $i\in\sI$, compute $\K'_i(n)$ from \eqref{EqA11c}, omitting
the duplicates mentioned in Remark~(2) above.

\item For every $i\in\sI$, compute $c_{i,n}$ and $h_{i,n}$ using \eqref{Eqci} and \eqref{Eqhi}, 
or using the equivalent formulas \eqref{EqciJ} and \eqref{EqhiJ} given below.

\item Compute $\B(n)$ with the formula 
\beql{EqFP1}
\B(n) ~=~ \frac{\K_{i-2c_{i,n}}'(n)+2}{b-1}\,,
\eeq
which follows from \eqref{hatBprop} and \eqref{EqS1}, 
and holds for any $i \in \sI$.

\item For every $i\in\sI$, compute $\K_i(n)$ using Theorem~\ref{ThS2}.
\item Finally, compute $\K(n)$ from \eqref{EqGTM}.
\end{enumerate}

Below we illustrate the computations for even $b \ge 4$, 
while in the next two sections we provide further information 
about bases $4$, $5$, $7$, and $10$.
Additional illustrations of the computation flow are given in~\ref{AppFlow}.

\vspace*{0.1in}
\noindent \textbf{Examples.}
We illustrate the computations using Theorems~\ref{ThBB} 
and \ref{ThS2} in the case when $b \ge 4$ is even and $n=2,3,4$.

For $n=2$, we find that $\B(2)= (2b-2)/(b-1)=2$, and
${\K}'_i(2) = 2b-4$, $c_{i,2}=1$ and $h_{i,2}=1$ for all $i$.
From this we obtain the values of $\K_i(2)$ that we saw in Theorem~\ref{ThKi2}.

For $n=3$,
we find that $\B(3)=(b^2+2b-3)/(b-1)=b+3$,
${\K}'_i(3)=b^2+2b-5$ and $c_{i,3}=1$ for all $i$,
and $h_{0,3}=2$, $h_{1,3}=h_{2,3}=h_{3,3}=\cdots = 1$.
From this we obtain the values of $\K_i(3)$ that we saw in Theorem~\ref{ThKi3}. 

For $n=4$, we find that ${\K}'_0(4) = 2b^2+2b-8$,
${\K}'_i(4) = 2b^2+2b-6$ for $i \ge 1$, $\B(4) = (2b^2+2b-4)/(b-1) = 2b+4$,
$c_{0,4}=2$, $c_{1,4}=1$, $c_{2,4}=3$, $c_{i,4}=1$ for $3 \le i \le b-2$, 
and all $h_{i,4}=2$.
Then 
\begin{align}\label{EqQQ4}
{}&{\K}_0(4)=2b^{2b+4}+b^2+2b-5\,, \nonumber \\
{}&{\K}_1(4)=b^{2b+4}+b^2+b-2\,, \nonumber \\
{}&{\K}_2(4) = 3b^{2b+4}+b^2+2b-4\,, \\
{}&{\K}_3(4)=b^{2b+4}+b^2+b\,, \nonumber \\
{}&{\K}_4(4)=b^{2b+4}+b^2+2\,, \nonumber 
\end{align}
and thus
\beql{EqK4}
{\K}(4) = {\K}_4(4) = b^{2b+4}+b^2+2\,.
\eeq
This confirms Kaprekar's conjecture of $10^{24}+102$ in base $10$.

Table~\ref{TabA1} summarizes the results from the recurrence
for $n \le 7$ and even bases $b \ge 4$ and odd bases $b \ge 7$.
(For smaller values of $b$, see Tables~\ref{TabK1}, 
\ref{TabK2}, \ref{TabK2a}.)

\begin{table}[htb]
$$
\begin{array}{|c|c|c||c|c|}
\hline
   & \multicolumn{2}{c|| }{\text{even~} b \ge 4}
   & \multicolumn{2}{c| }{\text{odd~} b \ge 7} \\
\cline{1-5}
n       &  \B(n) & \K(n)
   &  \B(n) & \K(n)   \\
\hline
1 &  - & 0
   & - & 0  \\
2 & 2 & b^2+1
   & 1 & b+1 \\
3 & b+3 & b^{b+3}+1
   & 2 & b^2+1  \\
4 & 2b+4 & b^{2b+4}+b^2+2
   & 3 & b^3+b+2  \\
5 & \frac{b^{b+3}+b^2+2b-4}{b-1} & b^{\B(5)}+b^2+2
   &b+3 & b^{b+3} + b+2    \\
6 & \frac{2 b^{b+3}+2b-4}{b-1} & b^{\B(6)}+b^{b+3}+2
   & 2b+3 & b^{2b+3}+b^2+2    \\
7 & \frac{b^{2b+4}+b^{b+3}+b^2+2b-5}{b-1} & b^{\B(7)}+b^{b+3}+2
   & b^2+2b+4 & b^{\B(7)}+b^2+2    \\
\hline
\end{array}
$$
\caption{ Values of $\B(n)$ and $\K(n)$ for $n \le 7$ and 
even bases $b \ge 4$,
odd bases $b \ge 7$. See also Fig.~\ref{fig:b4} and 
Figs.~\ref{fig:b6}--\ref{fig:b10} in~\ref{AppFlow}. }
\label{TabA1}
\end{table}

Let us define the set
\beql{EqJn}
J(n) := \{ j\in\sI\ \mid\ \min_{\ell\in\sI} \K_\ell\left(\Ceil{ \tfrac{n}{2} }\right)
+\K_{-\ell-2}\left(\Floor{ \tfrac{n}{2} }\right)\ \text{is attained at}\ \ell=j \}.
\eeq
Then $c_{i,n}$ and $h_{i,n}$ can be equivalently expressed as
\beql{EqciJ}
c_{i,n} ~=~ 
\text{~smallest~integer~} c\geq 1 \text{~such~that~} 
i-2c\in J(n) \text{~or~} 2c-i-2\in J(n)
\eeq
and
\beql{EqhiJ}
h_{i,n} = 
\begin{cases}
\Ceil{ \frac{n}{2} }, & \text{if}\ 2c_{i,n}-i-2\notin J(n);\\
\Floor{ \frac{n}{2} }, & \text{otherwise}\,.
\end{cases}
\eeq
In bases $b=2$ and $b=3$, we trivially have $J(n)=\sI=\{0\}$ for all $n\geq 2$.
In bases $b=4$ and $b\geq 6$, from Theorems~\ref{ThKi1}, \ref{ThKi2}, and \ref{ThK3} it can be verified that 
\beql{J23456}
\begin{split}
J(2) &=\sI, \\
J(3) &=\sI\setminus\{0\}, \\
J(4)=J(5)=J(6) &= \sI\setminus\{0,b-3\}. \\
\end{split}
\eeq
In particular, this implies that in bases $b=4$ and $b \ge 6$,
\begin{itemize}
\item $c_{i,2}=c_{i,3}=1$ and $h_{i,2}=h_{i,3}=1$ for all $i\in\sI$, with the exception of $h_{0,3}=2$ (as was already noted in the proof of Theorem~\ref{ThS2});
\item $c_{i,4}=c_{i,5}=c_{i,6}=1$, $h_{i,4}=h_{i,5}=2$ and $h_{i,6}=3$ for all $i\in\sI$, with the exception of $c_{0,4}=c_{0,5}=c_{0,6}=2$ and $c_{2,4}=c_{2,5}=c_{2,6}=3$.
\end{itemize}

For base $b=5$, we have $J(2)=J(4)=J(6)=\{0,2\}=\sI$ and $J(3)=J(5)=\{2\}$,\footnote{The apparent pattern does not continue here as $J(7)=\{0\}$.}
implying that $c_{i,2}=c_{i,3}=c_{i,4}=c_{i,5}=c_{i,6}=1$, $h_{i,2}=h_{i,3}=1$, $h_{i,4}=h_{i,3}=2$, and $h_{i,6}=3$ for all $i\in\sI$, with the exception of $h_{0,3}=2$ and $h_{0,5}=3$. More details can be found in Section~\ref{SecB5}.



\section{Quasi-positional representation for \texorpdfstring{$\K(n)$}{K(n)}}\label{SecQuasi}

Theorem~\ref{ThS2} suggests that $\K_i(n)$ can be expressed as a linear combination of terms
\beql{EqcB}
\cB(m) := b^{\B(m)} + 1
\eeq
and a single term $\K_{\beta}(1)$ for some $\beta\in\sI$. 
In this section we investigate properties of this representation and show that it resembles a conventional positional numeral system.


\begin{thm}\label{Thm_RepB}
For any $n\geq 1$ and $i\in\sI$, $\K_i(n)$ is uniquely represented as
\beql{EqRepK}
\K_i(n) = \alpha_1\cB(n_1) + \alpha_2\cB(n_2) + \dots + \alpha_t\cB(n_t) + \K_{\beta}(1)
\eeq
for some $t \ge 0$, where
\begin{itemize}
    \item $\alpha_1,\alpha_2,\dots,\alpha_{t}$ are integers from the interval $[1,b-1]$ if $b$ is even, or
the interval $[1,\tfrac{b-1}{2}]$ if $b$ is odd;
    \item $\beta\in\sI$; and
    \item if $t>0$, then $n_1 > n_2 > \dots > n_t$  with $n_j\in\left\{ \Floor{ \tfrac{n_{j-1}}{2} }, \Ceil{ \tfrac{n_{j-1}}{2} }\right\}$ 
    for $j=2,3,\dots,t$, and $n_t \in\{2,3\}$.
\end{itemize}
\end{thm}

\begin{proof} 
If $n=1$, we set $t:=0$ and $\beta:=i$.
If $n\geq 2$, we set $n_1:=n$. From Theorem~\ref{ThS2} it follows that $\K_i(n_1) = \alpha_1\cB(n_1) + \K_{i-2\alpha_1}(n_2)$, where $\alpha_1 := c_{i,n_1}$ and $n_2:=h_{i,n_1}\in\left\{ \Floor{ \tfrac{n_1}{2} }, \Ceil{ \tfrac{n_1}{2} }\right\}$. 
If $n_2>1$, we represent $\K_{i-2\alpha_1}(n_2)$ in a similar way, and continue the process until we get a representation \eqref{EqRepK},
where $n_j := h_{i-2(\alpha_1 + \alpha_2 + \dots + \alpha_{j-2}),n_{j-1}}$ and $\alpha_j := c_{i - 2(\alpha_1 + \alpha_2 + \dots + \alpha_{j-1}),n_j}$ for $j\geq 2$,
and $\beta:=i - 2(\alpha_1 + \alpha_2 + \dots + \alpha_t)$.
The properties of $\alpha_j$ and $n_j$ easily follow from this construction.

We prove that the representation \eqref{EqRepK} is unique by induction on $t$. For $t=0$, the statement follows from Theorem~\ref{ThKi1} as all $K_{\beta}(1)$ for $\beta\in\sI$ are distinct.
For $t\geq 1$, we let $k:=\K_i(n) - \alpha_1\cB(n_1)$ and show that $\alpha_1(b^{\B(n_1)} + 1) + k$ has the form~\eqref{EqStd}.
Indeed, if $t=1$, then 
\begin{itemize}
\item when $b$ is even, $k = K_{\beta}(1)\leq 2b-4$ (by Theorem~\ref{ThKi1}) and $\B(n_1)\geq \B(2)=2$ (by Theorem~\ref{ThK2}), implying that $k\leq 2b-4 \leq b^2 - b + 1\leq b^{\B(n_1)} - b + 1$
satisfying \eqref{Eqk};
\item when $b$ is odd, $k = K_{\beta}(1)\leq b-3$ (by Theorem~\ref{ThKi1}), $\alpha_1\leq \frac{b-1}2 < b-1$, and $\B(n_1)\geq \B(2)=1$ (by Theorem~\ref{ThK2}), 
implying that $k\leq b-3 \leq b\leq b^{\B(n_1)}$ satisfying \eqref{Eqk}.
\end{itemize}
Finally, if $t\geq 2$, then $n_1\geq 4$ and $n_1-n_2\geq 2$, implying by \eqref{EqS4} that $\B(n_1)\geq \B(n_2)+2$. Then by induction $k=\alpha_2(b^{\B(n_2)} + 1) + \cdots$ has the form~\eqref{EqStd} and since
$\alpha_2\leq b-1$, we get
$$k\leq \alpha_2(b^{\B(n_2)} + 1) + b^{\B(n_2)} \leq b^{\B(n_2)+1} + b - 1 \leq b^{\B(n_1)} - b + 1,$$
satisfying \eqref{Eqk}.
That is, we have proved that $\alpha_1(b^{\B(n_1)} + 1) + k$ has the form~\eqref{EqStd}, and then the uniqueness of the representation \eqref{EqRepK} follows from that of \eqref{EqStd}.
\end{proof}

Our proof suggests that the representation \eqref{EqRepK} can be viewed as the result of iterative application of Lemma~\ref{Lem_mkc}, which enables positional comparison of such representations.
In particular, for bounding purposes we will find it convenient to introduce a generic notation $O_{\cB}(n_1)$ for the right-hand side of \eqref{EqRepK} (when $t\geq 1$).

We will need the following lemma.

\begin{lem}\label{LemSum2K}
Let $n\geq 3$. Then for any $x,y\in\{1,2,\dots,n-1\}$ and any $i,j\in\sI$,
$$
\K_i(x) + \K_j(y) < \cB(n).
$$
\end{lem}

\begin{proof}
From Theorems~\ref{ThKi1} and \ref{ThK3}, it follows that for any $i\in\sI$
$$\K_i(1) \leq 2b-4 < \frac{b^2+1}{2} \leq \frac{1}{2}\cB(3) \leq \frac{1}{2}\cB(n).$$

For $x>1$ and any $i\in\sI$, from Theorems~\ref{Th_Ki_bounds} and \ref{ThBB}, it follows that for odd $b$
$$\K_{i}(x) < \frac{1}{2}b^{\B(x)+1} < \frac{1}{2}b^{\B(n-1)+1} \leq \frac{1}{2}b^{\B(n)} < \frac{1}{2}\cB(n),$$
while for even $b$ 
$$\K_{i}(x) < b^{\B(x)+1} < b^{\B(n-1)+1} \leq \frac{1}{2}b^{\B(n-1)+2} \leq \frac{1}{2}b^{\B(n)} < \frac{1}{2}\cB(n).$$

Hence in all cases,
$$\K_{i}(x) + \K_j(y) < \frac{1}{2}\cB(n) + \frac{1}{2}\cB(n) = \cB(n).$$
\end{proof}

The following theorem shows that the quasi-positional representations for $\K(n)$ in bases $b_1=2m$ 
and $b_2=4m-1$ are essentially the same.
We use superscripts $(b_1)$ and $(b_2)$ to distinguish 
between the bases.


\begin{thm}\label{Thm_b1b2}
For an integer $m\geq 1$, let $b_1:=2m$ and $b_2:=4m-1$.
We identify the additive groups
$\sI^{(b_1)}$ (consisting of the residues modulo $2m-1$) 
and $\sI^{(b_2)}$ (consisting of the even residues modulo $2(2m-1)$).\footnote{More formally, we may define an isomorphism 
of additive groups $\pi:\ \sI^{(4m-1)} \rightarrow \sI^{(2m)}$ by
$\pi([i\bmod 2(2m-1)]) = [i\bmod (2m-1)]$.}
Then for any $n\geq 1$ and 
any $i\in \sI = \sI^{(b_1)} \cong \sI^{(b_2)}$, $\K_i^{(b_1)}(n)$ has a 
representation in the form~\eqref{EqRepK}
$$\K_{i}^{(b_1)}(n) = \alpha_1\cB^{(b_1)}(n_1) + 
\alpha_2\cB^{(b_1)}(n_2) + \dots + \alpha_t\cB^{(b_1)}(n_t) + 
\K_{\beta}^{(b_1)}(1)$$
if and only if $\K_i^{(b_2)}(n)$ has a representation:
$$\K_i^{(b_2)}(n) = \alpha_1\cB^{(b_2)}(n_1) + 
\alpha_2\cB^{(b_2)}(n_2) + \dots + \alpha_t\cB^{(b_2)}(n_t) + 
\K_{\beta}^{(b_2)}(1).$$
\end{thm}

\begin{proof}
The proof is by induction on $n$. For $n=1$, the statement follows 
from Theorem~\ref{ThKi1}, which gives 
$\K_{\beta}^{(b_1)}(1) = \K_{\beta}^{(b_2)}(1)$ for all $\beta\in\sI$. 

Let $n\geq 2$. For any $i\in \sI$, Theorem~\ref{ThS2} gives
$$\K_i^{(b_1)}(n) = c^{(b_1)}_{i,n}\cB^{(b_1)}(n) + \K^{(b_1)}_{i-2c^{(b_1)}_{i,n}}(h^{(b_1)}_{i,n})$$
and
$$\K_{i}^{(b_2)}(n) = c^{(b_2)}_{i,n}\cB^{(b_2)}(n) + \K^{(b_2)}_{i-2c^{(b_2)}_{i,n}}(h^{(b_2)}_{i,n}).$$
The statement will follow by induction if we show that 
$c^{(b_1)}_{i,n} = c^{(b_2)}_{i,n}$ and 
$h^{(b_1)}_{i,n} = h^{(b_2)}_{i,n}$.
In each base, the values of $c_{i,n}$ and $h_{i,n}$ are 
entirely determined by the set of indices $J(n)\subset\sI$ 
defined in~\eqref{EqJn}. It is therefore sufficient to show that 
$J^{(b_1)}(n)=J^{(b_2)}(n)$.
This equality holds for $n\leq 6$ as established 
by~\eqref{J23456} (notice that $b_1 - 3 = b_2 - 3$ as 
residues in $\sI$). For the rest of the proof we assume that $n\geq 7$.

The set $J(n)$ consists of the indices $j$ producing the 
minimum of 
$\K_j\left(\Ceil{\tfrac{n}{2}}\right) + 
\K_{-j-2}\left(\Floor{\tfrac{n}{2}}\right)$. 
Hence we need to show that comparing two such sums  
gives the same result (``$<$'',  ``$=$'', or ``$>$'')  in the two bases.
More generally, we will prove the following statement:
\begin{itemize}
    \item[$(\star)$]
For any positive integers $u,v,u',v'$ less than $n$ and any $i,j,i',j'\in\sI$,
the result of comparing $X^{(b)}:=\K^{(b)}_i(u) + \K^{(b)}_j(v)$ and $Y^{(b)}:=\K^{(b)}_{i'}(u') + \K^{(b)}_{j'}(v')$ is the same in bases $b=b_1$ and $b=b_2$. 
\end{itemize}
Without loss of generality we assume that $\max\{u,v,u',v'\}=u$,
and prove the statement by induction on $u$.
For $u=1,2$,  $(\star)$ follows from Theorems~\ref{ThKi1} and \ref{ThKi2}.
Suppose now that $2<u<n$. 

Let $(x,k)\in\{ (u,i), (v,j), (u',i'), (v',j') \}$. 
If $x=u$, we have $x<n$, and by induction on $n$,
$c_{k,x}^{(b_1)} = c_{k,x}^{(b_2)} =: \alpha_{k,x}$ and 
$h_{k,x}^{(b_1)} = h_{k,x}^{(b_2)} =: w_{k,x}$. 
Then by Theorem~\ref{ThS2}, 
\beql{EqKbkx}
\K^{(b)}_k(x) = \alpha_{k,x}\cB^{(b)}(u) + 
\K^{(b)}_{k-2\alpha_{k,x}}(w_{k,x}),\qquad b\in\{b_1,b_2\}.
\eeq
If $x<u$, we obtain the same representations \eqref{EqKbkx} by setting $\alpha_{k,x} := 0$ and $w_{k,x}:=x$.
Notice that in all cases, we have $w_{k,x} < u$.

Using the representations \eqref{EqKbkx}, we get
$$X^{(b)} = (\alpha_{i,u} + \alpha_{j,v})\cB^{(b)}(u) + \K^{(b)}_{i-2\alpha_{i,u}}(w_{i,u}) + \K^{(b)}_{j-2\alpha_{j,v}}(w_{j,v})$$
and
$$Y^{(b)} = (\alpha_{i',u'} + \alpha_{j',v'})\cB^{(b)}(u) + \K^{(b)}_{i'-2\alpha_{i',u'}}(w_{i',u'}) + \K^{(b)}_{j'-2\alpha_{j',v'}}(w_{j',v'}).$$

By Lemma~\ref{LemSum2K}, the sum of two $\K$'s in $X^{(b)}$ is smaller than $\cB^{(b)}(u)$, and so is the sum of two $\K$'s in $Y^{(b)}$.
Hence, if $\alpha_{i,u} + \alpha_{j,v}$ and $\alpha_{i',u'} + \alpha_{j',v'}$ are not equal,
then their comparison completely determines the result of comparison of $X^{(b)}$ and $Y^{(b)}$ in each base $b$. 

On the other hand, if $\alpha_{i,u} + \alpha_{j,v} = \alpha_{i',u'} + \alpha_{j',v'}$, then the result of comparison of $X^{(b)}$ and $Y^{(b)}$ is determined by comparison of 
$\K^{(b)}_{i-2\alpha_{i,u}}(w_{i,u}) + \K^{(b)}_{j-2\alpha_{j,v}}(w_{j,v})$
and $\K^{(b)}_{i'-2\alpha_{i',u'}}(w_{i',u'}) + \K^{(b)}_{j'-2\alpha_{j',v'}}(w_{j',v'})$, which is the same for $b=b_1$ and $b=b_2$ by induction (since all $w$'s are smaller than $u$).
This completes the proof of statement $(\star)$, which further implies that $J(n)$ is the same for $b=b_1$ and $b=b_2$, and thus proves the theorem by induction for all $n\geq 1$.
\end{proof}

Theorem~\ref{Thm_b1b2} explains why the flow-charts 
in~\ref{AppFlow} for
$b=2$ and $b=3$ (Figs.~\ref{fig:b2} and \ref{fig:b3})
are the same apart from the labels, 
as are the flow-charts for $b=4$ and $b=7$
(Figs.~\ref{fig:b4} and \ref{fig:b7}).



\section{\texorpdfstring{$\K(n)$}{K(n)} for bases \texorpdfstring{$4$}{4},
\texorpdfstring{$5$}{5}, and \texorpdfstring{$7$}{7}}\label{SecB5}

\begin{table}[htb]
$$
\begin{array}{|c|ccccccc|}
\hline
n       &  \K'_0(n) & \B(n) & h_{0,n} & h_{2,n} & \K_0(n) & & \K_2(n)   \\
\hline
1 & - & - & - & -   & \mathbf{0} & \DR & 2 \\
2 & 2 & 1 & 1 & 1       & 5+3 & \DR & \mathbf{5+1} \\
3 & 6 & 2 & 2 & 1 &5^2+7 & \DL   & \mathbf{5^2+1}  \\
4 & 14 & 4 & 2 & 2   & \mathbf{5^4+7} & \DR & 5^4+9 \\
5 & 34 & 9 & 3 & 2  & 5^9+27 & \DL & \mathbf{5^9+9}  \\
6 & 58 & 15 & 3 & 3      & \mathbf{5^{15}+27} & \DL & 5^{15}+33  \\
7 & 658 & 165 & 3 & 4       & \mathbf{5^{165}+27} & \DR & 5^{165}+633 \\
8 & 1266 & 317 & 4 & 4   & 5^{317}+635 & \DR & \mathbf{5^{317}+633} \\
9 & 5^9+5^4+16 & 488442 & 5 & 4      & 5^{\B(9)}+5^9+10 & \DL &
\mathbf{5}^{\Bbi \mathbf{(9)}}\mathbf{+5^4+8} \\
10 & 2\cdot 5^9+36 & 976572 & 5 & 5 & \mathbf{5}^{\Bbi \mathbf{(10)}}\mathbf{+5^9+10} & \DL & 5^{\B(10)}+5^9+28  \\      \hline
\end{array}
$$
\caption{Base $5$: $\K(n)$ (shown in bold font) is the smaller of the entries in the last two columns.
The meaning of the arrows is explained in the text. See also Fig.~\ref{fig:b5} in~\ref{AppFlow}.}
\label{TabK5}
\end{table}

We discuss base $5$ first, since this turns out to be simpler 
than bases $4$ or $7$.

For $b=5$, the index set is $\sI = \{0,2\}$.  From
\eqref{EqKp1}, there is only one $\K'_i(n)$ to
consider, namely 
\beql{Eqh4}
\K'_0(n) ~=~
 \min ~
\left\{
\K_0\left(\Ceil{ \tfrac{n}{2} }\right)+\K_{2}\left(\Floor{ \tfrac{n}{2} }\right)
,~
\K_0\left(\Floor{ \tfrac{n}{2} }\right)+\K_{2}\left(\Ceil{ \tfrac{n}{2} }\right)
 \right\}\,.
\eeq
Then $\B(n)= (\K'_0(n)+2)/4$,  $c_{0,n}=c_{2,n} = 1$ for all $n$,
\beql{Eqh5}
h_{0,n} ~=~
\begin{cases}
\Ceil{ \frac{n}{2} },
  &\text{if~~}
 \K_{0}\left(\Floor{ \tfrac{n}{2} }\right) +
 \K_{2}\left(\Ceil{ \tfrac{n}{2} }\right)
 <
 \K_{0}\left(\Ceil{ \tfrac{n}{2} }\right) +
 \K_{2}\left(\Floor{ \tfrac{n}{2} }\right), \\
 \Floor{ \frac{n}{2} }\,,
&\text{otherwise}\,, \\
\end{cases}
\eeq
and $h_{2,n}= n-h_{0,n}$. Also
\beql{Eqdagger}
\K_0(n) ~=~ 5^{\B(n)}+\K_2(h_{0,n})+1, ~~
\K_2(n) ~=~ 5^{\B(n)}+\K_0(h_{2,n})+1\,,
\eeq
and $\K(n) = \min \{ \K_0(n), \K_2(n) \}$.
The initial values of these variables are shown in Table~\ref{TabK5}.
The value of $\K(n)$ is shown in bold font (the
first $100$ values of $\B(n)$ and $\K(n)$ can be found in 
\seqnum{A230868}\label{A230868} and \seqnum{A230867}\label{A230867b}). 
The symbol in the
penultimate column of the table indicates the choice made
in~\eqref{Eqh4} when calculating $\K'_0(n)$ for odd $n$.
An arrow $\DR$ in row $i$ indicates that 
$\K'_0(2i+1) =
\K_0(i)+\K_{2}(i+1)$,
while an arrow $\DL$ indicates that
$\K'_0(2i+1) = \K_0(i+1)+\K_{2}(i)$.\footnote{The arrows are intended to suggest,
for the four elements $\K_0(i), \K_2(i), \K_0(i+1), \K_2(i+1)$,
whether it is better to pair up the North West and South East
entries, or the North East and South West
entries.}

The values of the generalized Thue--Morse sequence
$\tau(n)$ (see \eqref{EqTAU}) are shown in Table~\ref{TabTM2}.
\begin{table}[htbp]
$$
\begin{array}{|c|rrrrrrrrrrrrrrrr|}
\hline
n & 1 & 2 & 3 & 4 & 5 & 6 & 7 & 8 & 9 & 10 & 11 & 12 & 13 & 14 & 15 & 16  \\
\hline
\tau(n) & 0 & 1 & 1 & 0 & 1 & 0 & 0 & 1 & 1 & 0 & 0 & 1 & 0 & 1 & 1 & 0 \\
\hline
\end{array}
$$
\caption{Initial values of $\tau(n)$ (this is essentially the 
Thue--Morse sequence
\seqnum{A010060}\label{A010060b}). }
\label{TabTM2}
\end{table}

In this case $\tau(n)$ actually {\em is} 
the classical Thue--Morse sequence, 
except shifted by one step.\footnote{The classical sequence is $\tau(n-1)$.}
We prove this in the next theorem.

\begin{thm}\label{ThTM}
For $n \ge 1$, 
$$
\tau(n) = \begin{cases}
\tau(\Ceil{ \tfrac{n}{2} }), & \text{if $n$ is odd}\,;\\ 
1-\tau(\tfrac{n}{2}), & \text{if $n$ is even}\,.
\end{cases}
$$
\end{thm}

\begin{proof} (Sketch.)
The basis for the inductive proof are the following observations.

(i) If $n$ is even, then \eqref{Eqdagger} implies that 
$\K_0(n) < \K_2(n)$ if and only if $\K_2\left(\tfrac{n}{2}\right)  < \K_0\left(\tfrac{n}{2}\right)$,
and hence  $\tau(n) = 1 - \tau(\tfrac{n}{2})$.

(ii) Suppose on the other hand that $n=2i+1$ is odd. 
There are two possibilities. If
\beql{Eq31}
\K_0(i)+\K_2(i+1) ~<~ \K_0(i+1)+\K_2(i)
\eeq
(the $\DR$ case in Table~\ref{TabK5}), then $\K_2(i+1)<\K_0(i+1)$,
$\tau(i+1)=2$, $\K(n) = \K_2(n)$, 
and hence $\tau(n)=1=\tau(\Ceil{ \tfrac{n}{2} })$. If the inequality in~\eqref{Eq31}
is reversed, we similarly find that $\tau(n)=2=\tau(\Ceil{ \tfrac{n}{2} })$.

For the induction to work, we need to also show that 
the values of $\B(n)$, $\K_0(n)$, and $\K_2(n)$ are
considerably larger than the values of $\B(n-1)$, $\K_2(n-1)$, 
and $\K_0(n-1)$, respectively, but this follows from
Theorem~\ref{Th_Ki_bounds} and \eqref{EqS4}.
We omit the details of the proof.
\end{proof}

The situation is more complicated in base $4$. 
Here the index set is $\sI = \{0,1,2\}$ (modulo 3), 
$\K(n)$ is the minimum of the three
terms $\K_0(n)$, $\K_1(n)$, $\K_2(n)$, 
and so is specified by a ternary sequence $\tau(n)\in\sI$.
The first $100$ values of $\B(n)$ and $\K(n)$ can be found 
in \seqnum{A230637}\label{A230637}
and \seqnum{A230638}\label{A230638}, 
and the first $100$ terms of $\tau(n)$ are: 
\begin{align}\label{EqTM4}
&0,2,2,1,1,1,2,0,2,0,2,0,0,1,1,1,1,1,1,1,1,1,1,1,1,1,2,0,2,0,2,0,2,0,2, \nonumber \\
&0,2,0,2,0,2,0,2,0,2,0,2,0,2,0,2,0,0,1,1,1,1,1,1,1,1,1,1,1,1,1,1,1,1,1, \nonumber \\
&1,1,1,1,1,1,1,1,1,1,1,1,1,1,1,1,1,1,1,1,1,1,1,1,1,1,1,1,1,1, \ldots
\end{align}
(\seqnum{A239110}\label{A239110}). 
If we write $(2,0)^{13}$ to denote a run of $13$
copies of $2,0$, etc., then the first $1000$ terms
of this sequence are
\beql{EqTM4a}
0, 2^2, 
1^3, (2,0)^3, 0,
1^{13}, (2,0)^{13}, 0, 
1^{53}, (2,0)^{53}, 0,
1^{213}, (0,2)^{213}, 0,
1^{\ge 147}\,,
\eeq
which suggests that after the initial three terms,
there is a repeating pattern
$$
1^{\delta_j}, (2,0)^{\delta_j}, 0\,,
$$
where $\delta_j = \tfrac{10\cdot 4^j-1}{3}$ 
(see \seqnum{A072197}\label{A072197}).
The next theorem shows that this pattern continues for ever.


\begin{thm}
Let $b=4$ and let $d\geq 0$ be an integer. For an integer $n$, 
\begin{enumerate}[(i)]
    \item if $\tfrac{10\cdot 4^{d}+2}{3}\leq n<\tfrac{20\cdot 4^{d}+1}{3}$, then
\[
\begin{split}
K_0(n) &= 2\cB(n) + \cB(\Floor{\tfrac{n}{2}}) + 
O_{\cB}(\Floor{\tfrac{n}{4}})\,,\\
K_1(n) &= \cB(n) + \cB(\Floor{\tfrac{n}{2}}) + 
O_{\cB}(\Floor{\tfrac{n}{4}})\,,\\
K_2(n) &= 3\cB(n) + \cB(\Floor{\tfrac{n}{2}}) + 
O_{\cB}(\Floor{\tfrac{n}{4}})\,,
\end{split}
\]
and thus $\tau(n) = 1$;

\item if $\tfrac{20\cdot 4^{d}+2}{3}\leq n<\tfrac{10\cdot 4^{d+1}-1}{3}$, then
\[
\begin{split}
K_0(n) &= \cB(n) + \cB(\Ceil{\tfrac{n}{2}}) + 
O_{\cB}(\Ceil{\tfrac{n}{4}})\,,\\
K_1(n) &= 2\cB(n) + (2-(n\bmod 2))\cB(\Floor{\tfrac{n}{2}}) + 
O_{\cB}(\Floor{\tfrac{n}{4}})\,,\\
K_2(n) &= \cB(n) + (2-(n\bmod 2))\cB(\Floor{\tfrac{n}{2}}) +
O_{\cB}(\Floor{\tfrac{n}{4}})\,,
\end{split}
\]
and thus $\tau(n) = 2$ when $n$ is odd, and $\tau(n) = 0$ when $n$ is even;

\item if $n = \tfrac{10\cdot 4^{d+1}-1}{3}$, then
\[
\begin{split}
K_0(n) &= \cB(n) + \cB(\Floor{\tfrac{n}{2}}) + 
O_{\cB}(\Ceil{\tfrac{n}{4}})\,,\\
K_1(n) &= 2\cB(n) + \cB(\Ceil{\tfrac{n}{2}}) + 
O_{\cB}(\Floor{\tfrac{n}{4}})\,,\\
K_2(n) &= \cB(n) + \cB(\Ceil{\tfrac{n}{2}}) +
O_{\cB}(\Floor{\tfrac{n}{4}})\,,
\end{split}
\]
and thus $\tau(n) = 0$.
\end{enumerate}
\end{thm}

\begin{proof} We prove the statement by induction on $d$. For $d=0$ (i.e., $4\leq n\leq 13$), the statement can be verified directly (e.g., see Fig.~\ref{fig:b4}).
Let $d>0$.

(i) Let $n$ belong to the interval 
$\tfrac{10\cdot 4^{d}+2}{3}\leq n<\tfrac{20\cdot 4^{d}+1}{3}$. 
Then both $m = \Ceil{\tfrac{n}{2}}$ and $n-m=\Floor{\tfrac{n}{2}}$ 
are in the interval (ii). 
If $n$ is even, the values of $K_j(m) + K_{-2-j}(m)$ are
\[
\begin{split}
K_0(m) + K_1(m) &= 3\cB(m) + O_{\cB}(\Ceil{\tfrac{m}{2}})\,,\\
K_2(m) + K_2(m) &= 2\cB(m) + O_{\cB}(\Ceil{\tfrac{m}{2}})\,,
\end{split}
\]
implying that $J(n) = \{2\}$. 
If $n$ is odd, the values of $K_j(m) + K_{-2-j}(m-1)$ are
\[
\begin{split}
K_0(m) + K_1(m-1) &= \cB(m) + 2\cB(m-1) + O_{\cB}(\Ceil{\tfrac{m}{2}})\,,\\
K_1(m) + K_0(m-1) &= 2\cB(m) + \cB(m-1) + O_{\cB}(\Ceil{\tfrac{m}{2}})\,,\\
K_2(m) + K_2(m-1) &= \cB(m) + \cB(m-1) + O_{\cB}(\Ceil{\tfrac{m}{2}})\,,
\end{split}
\]
also implying that $J(n)=\{2\}$.
From \eqref{EqciJ} and \eqref{EqhiJ}, it follows that $c_{0,n}=2$, $c_{1,n}=1$, $c_{2,n}=3$ and $h_{0,n}=h_{1,n}=h_{2,n}=\Floor{\tfrac{n}{2}}$. Statement (i) now follows by induction from formula~\eqref{EqS2}.

Statements (ii) and (iii) are proved similarly. We omit the details.
\end{proof}

Theorem~\ref{Thm_b1b2} implies that essentially the same sequence $\left(\tau(n)\right)_{n\geq 1}$  arises in base $7$,
with the only difference being that $1$s and $2$s are interchanged.



\section{\texorpdfstring{$\K(n)$}{K(n)} for base \texorpdfstring{$10$}{10}}\label{SecB10}

In base $b=10$, the case studied by Kaprekar and others,
the index set is $\sI = \{0,1,2,\ldots,8\}$ (modulo 9), 
and, from 
\eqref{EqKp1}, there are five distinct $\K'_i(n)$, namely
$\K'_0(n)=\K'_7(n)$, $\K'_1(n)=\K'_6(n)$, $\K'_2(n)=\K'_5(n)$, 
$\K'_3(n)=\K'_4(n)$, and $\K'_8(n)$.
There are nine variables $c_{i,n}$, $h_{i,n}$, and $\K_i(n)$, 
with $0 \le i \le 8$.
The values of $\K_i(n)$ for $n \le 7$ are shown in Table~\ref{TabK10a}.
Then
\beql{Eq10e}
\K(n) ~=~ \min_{0 \le i \le 8} \K_i(n) ~=~ 
10^{\B(n)} + \text{~terms~of~smaller~order}\,.
\eeq
We have already seen $\B(n)$ and $\K(n)$ for $n \le 7$ in Table~\ref{TabK1}.
Tables~\ref{TabK10b} and \ref{TabK10c} extend these values to $n=16$,
going far enough that we can see -- and confirm! --
the values for $\K(4), \ldots, \K(8)$, and $\K(16)$ found
by Kaprekar and Narasinga Rao more than
fifty years ago (see the discussion in the Introduction).
The first $100$ terms of these two sequences can be seen in entries
\seqnum{A230857}\label{A230857} and \seqnum{A006064}\label{A006064c} in \cite{OEIS}.

\begin{table}[htb]
\footnotesize
$$
\begin{array}{|c|ccccc|}
\hline
      &  i=0 & i=1 & i=2 & i=3 & i=4 \\
\hline
\K_i(1) & \mathbf{0} & 10 & 2  & 12 & 4 \\
\K_i(2) & 117 & 109 & \mathbf{101} & 111 & 103 \\
\K_i(3) 
& 10^{13} + 116
& 10^{13} + 9
& \mathbf{10^{13} + 1}
& 10^{13} + 11
& 10^{13} + 3 \\
\K_i(4)  
& 2 \cdot 10^{24} + 115 
& 10^{24} + 108 
& 3 \cdot 10^{24} + 116
& 10^{24} + 110 
& \mathbf{10^{24} + 102}  \\
\K_i(5) 
& 2 \cdot 10^{\B(5)}+115
& 10^{\B(5)}+108
& 3 \cdot 10^{\B(5)}+116
& 10^{\B(5)}+108
& \mathbf{10^{\B(5)}+102} \\
\hline
\hline
       &  i=5 & i=6 & i=7 & i=8 &  \\
\hline
\K_i(1) & 14 & 6 & 16 & 8 & \\
\K_i(2) & 113 & 105 & 115 & 107 &  \\
\K_i(3)
& 10^{13} + 13
& 10^{13} + 5
& 10^{13} + 15
& 10^{13} + 7
& \\
\K_i(4)  
& 10^{24} + 112 
& 10^{24} + 104 
& 10^{24} + 114 
& 10^{24} + 106
&  \\ 
\K_i(5) 
& 10^{\B(5)}+112
& 10^{\B(5)}+104
& 10^{\B(5)}+114
& 10^{\B(5)}+106 
& \\
\hline
\end{array}
$$
\normalsize
\caption{Base $10$: $\K_i(n)$ for $n \le 5$; 
the value of $\K(n)$ (\seqnum{A006064}\label{A006064d}) is shown in bold font.
In the $\K_i(5)$ rows, $\B(5) = (10^{13}+116)/9 = 1111111111124$ 
as in Table~\ref{TabK1}. See also Fig.~\ref{fig:b10} in~\ref{AppFlow}.} 
\label{TabK10a}
\end{table}

\begin{table}[htb]
$$
 \begin{array}{|c|c|} 
\hline
n       &  \B(n) \\
\hline
8 & (2 \cdot 10^{24}+214)/9 \\
9 & (10^{(10^{13}+116)/9}+10^{24}+214)/9 \\
10 & (2 \cdot 10^{(10^{13}+116)/9}+214)/9 \\
11 & (10^{(2 \cdot 10^{13}+16)/9}+10^{(10^{13}+116)/9}+10^{13}+114)/9 \\
12 & (2 \cdot 10^{(2 \cdot 10^{13}+16)/9}+2 \cdot 10^{13}+14)/9 \\
13 & (10^{\B(7)} +10^{(2 \cdot 10^{13}+16)/9}+2 \cdot 10^{13}+14)/9 \\
14 & (2 \cdot 10^{\B(7)} +2 \cdot 10^{13}+14)/9 \\
15 & (10^{\B(8)} +10^{\B(7)} + \B(7) - 2)/9 \\
16 & (2 \cdot 10^{\B(8)} + \B(8) -2)/9 \\
\hline
\end{array}
$$
\caption{Base $10$: $\B(n)$ for $8 \le n \le 16$,
extending~Table~\ref{TabK1};
$\B(7)=  (10^{24}+10^{13}+115)/9$. See also Fig.~\ref{fig:b10} in~\ref{AppFlow}. }
\label{TabK10b}
\end{table}

 \begin{table}[htb]
$$
\begin{array}{|c|c|}
\hline
n       &  \K(n) \\
\hline
8 & 10^ {\B(8)} +10^{24}+103 \\
9 & 10^ {\B(9)} +10^{24}+103 \\
10 & 10^ {\B(10)} +10^{(10^{13}+116)/9}+103 \\
11 & 10^ {\B(11)} +10^{(10^{13}+116)/9}+103 \\
12 & 10^ {\B(12)} +10^{(2 \cdot 10^{13}+16)/9}+10^{13}+3 \\
13 & 10^ {\B(13)} +10^{(2 \cdot 10^{13}+16)/9}+10^{13}+3 \\
14 & 10^ {\B(14)} +10^{(10^{24}+10^{13}+115)/9}+10^{13}+3 \\
15 & 10^ {\B(15)} +10^{(10^{24}+10^{13}+115)/9}+10^{13}+3 \\
16 & 10^ {\B(16)} +10^{(2 \cdot 10^{24}+214)/9}+10^{24}+104 \\
\hline
\end{array}
$$
\caption{Base $10$: $\K(n)$ for $8 \le n \le 16$,
extending Table~\ref{TabK1}. See also Fig.~\ref{fig:b10} in~\ref{AppFlow}.}
\label{TabK10c}
\end{table}

The first $100$ terms of the base $10$ generalized Thue--Morse sequence 
$\tau(n)$ are as follows:
 \begin{align}\label{EqTM10}
& 0,2,2,4,4,4,4,6,6,6,6,6,6,6,6,8,8,8,8,8,8,8,8,8,8,8,8,8,8,8,8,1,1,1, \nonumber \\
& 1,1,1,1,1,1,1,1,1,1,1,1,1,1,1,1,1,1,1,1,1,1,1,1,1,1,1,1,8,3,8,3,8,3, \nonumber \\
& 8,3,8,3,8,3,8,3,8,3,8,3,8,3,8,3,8,3,8,3,8,3,8,3,8,3,8,3,8,3,8,3, \dots \nonumber \\
\end{align}
(\seqnum{A239896}\label{A239896}). 
This can be rewritten as
$0, 2^2, 4^4, 6^8, 8^{16}, 1^{31}, (8,3)^{\ge 19}$,
but now, unlike the base $4$ case, there is no obvious pattern.



\section{Growth of \texorpdfstring{$\K(n)$}{K(n)}}\label{SecTow}
In this section we discuss the rate of growth of $\K(n)$
for a fixed $b$. 

The following theorem generalizes the inequalities \eqref{EqI2}
and \eqref{EqI3}. It implies that,
for any base $b$, $\{\K(n), n\ge 1\}$ is a sequence
of exponential type (cf. Lemma~\ref{Lem4}).


\begin{thm}\label{ThDouble}
For $b \ge 2$ and $n \ge 1$,
\beql{EqDouble1}
\K(n+1) ~>~ b\K(n)\,,
\eeq
except that for odd $b \ge 5$, we only have
\beql{EqDouble2}
\K(3) ~>~ (b-1)\K(2)\,.
\eeq
\end{thm}
\begin{proof}
For $n=1$ and $2$, the statement can be verified 
directly from Theorems~\ref{ThK2} and~\ref{ThK3}.

For $n=3$ and $b=3$ or $5$, the statement follows from Table ~\ref{TabK1}.
For $n=3$ and odd $b \ge 7$, 
Corollary~\ref{Cor20}, Theorem~\ref{ThK2}, and~\eqref{EqE4} imply that
$\K(4) \ge b^{\B(4)} + 1 + \K(2) = b^3 + b + 2.$
Since $\K(3)=b^2+1$ (by Theorem~\ref{ThK3}), we have $\K(4) \ge b \K(3)$.

For $n \ge 4$ and any $b$, as well as for $n=3$ and even $b$, 
the inequality~\eqref{EqS4} and Theorem~\ref{Th_Ki_bounds} imply
$\K(n+1) \ge b^{\B(n+1)} \ge b^{\B(n)+2} \ge b \K(n)$.
\end{proof}

Since $\K(n)$ grows rapidly, it is appropriate to
describe its value by a tower of exponentials.
For $b \ge 2$,
any number $u \ge 1$ can be written in a unique way as a ``tower"
\beql{EqTowU}
u ~=~ b^{  b^{  \Snj \cdot ^{ \Snj \cdot ^{ \Snj \cdot  ^{\Snj b ^{\Snj \Om } }  } }   }  } \,,
\eeq
with $0 < \Om \le 1$. If this tower contains $h-1$ $b$'s and one $\Om$,
we call $h$ the base $b$ {\em height} of $u$, denoted by $\Ht(u)$.
Then $\Ht(u)$ is one more than the number of times
one has to take logarithms to the base $b$ of $u$ until reaching a number
$\Om$ $\le 1$.

Examination of the data in Tables~\ref{TabK1}, \ref{TabK2}, \ref{TabK2a}, \ref{TabK5},
\ref{TabK10c} (and in the more extended tables in \cite{OEIS})
 suggests the following conjecture.
\begin{conj}\label{ConjTow} It appears that:

(i) If $b=2$ and $n \ge 2$, then $\Ht(n)=\Ceil{\log_2(n)}+3$;

(i) If $b=3$ and $n \ge 3$, then $\Ht(n)=\Ceil{\log_2(n/5)}+4$;

(iii) If $b \ge 4$ is even and $n \ge 2$, then $\Ht(n)=\Ceil{\log_2(n)}+2$;

(iii) If $b \ge 5$ is odd and $n \ge 2$, then $\Ht(n)=\Ceil{\log_2(n)}+1$.
\end{conj}

For example, in base $b=10$, 
\beql{EqTow10_3}
\K(3) ~=~ 10^{13}+1 ~=~  10^{  10^{  \Snj 10 ^{ 0.04686\ldots } }  } \,, 
\eeq
which has height $4$. The heights of $\K(2)$ through $\K(16)$ in base $10$ 
(see \eqref{EqKb10} and Tables~\ref{TabK1}, \ref{TabK10c})
are $3,4,4,5,5,5,5,6,6,6,6,6,6,6,6$, respectively, in agreement
with the conjecture.

There are two reasons for believing the conjecture.
First, it is true in every case that we have checked.
Second, from Section~\ref{SecBb},
$\K(n)$ is very roughly equal to
\beql{EqConj2}
b^{ (\K(\Ceil{ \tfrac{n}{2} }) + \K(\Floor{ \tfrac{n}{2} }) )/(b-1)}\,,
\eeq
which suggests that the height of the tower for $\K(n)$ 
is one greater than the height of the tower for
$\K(\Ceil{ \tfrac{n}{2} })$, which would leads to the formulas in the conjecture.
However, two difficulties arise when
trying to make this argument rigorous.
One is the fact that if $u$ in \eqref{EqTowU} has height $h$, and $\Om(u)$ is very
close to $1$, $b^u$ can have height $h+2$ instead of $h+1$.
This seems not to happen with $\K(n)$, but we cannot rule out
that possibility, even for base $2$.
The second difficulty is that~\eqref{EqConj2} ignores
the choices that must be made (for $b \ge 4$)
among the $\K_i(n)$
when determining $\K(n)$. 

The following example shows the first of these difficulties
in a simpler setting.
Consider the sequence defined by the recurrence
\beql{EqSimpler}
a(1)=0, ~ a(n) = 
2^{a(\Ceil{ \tfrac{n}{2} })+a(\Floor{ \tfrac{n}{2} })}\quad \text{for~} n\ge 2\,.
\eeq
This is similar to the recurrence for $\K(n)$ in base $2$ given in 
\eqref{EqKK2} and \eqref{EqB2}, except that the additive terms
on the right-hand sides of those equations are missing.
The initial values of $a(n)$ for $n=1,2,3,\ldots$  are
$$
0,1,2,4,8,16,64,2^8,2^{12},2^{16},2^{24},
2^{32}, 2^{80}, 2^{128}, 2^{320}, 2^{512}, 2^{4352}, \ldots \,,
$$
(\seqnum{A230863}\label{A230863}). The heights of $a(n)$ for $n=2,...,10$ are 
$1,1,2,3,4,4,5,5,5,5$. 
For $11 \le n \le 40$, 
if $9\cdot 2^{i-1} < n \le 9\cdot 2^i$ then $\Ht(a(n)) = i+5$,
although here we do not know if this will
hold for all $n$.\footnote{The difficulty lies in the fact that terms $a(n)$, for which the top entry
in the tower, $\Om$, is 1 or very close to 1, could disrupt the pattern of the heights.}

For small values of $n$, of course, there is no difficulty
in computing the height of $\K(n)$.
From Theorems~\ref{ThK2} and \ref{ThK3}, for example,
we have $\Ht(\K(2)) = 4$ if $b=2$, or $3$ if $b \ge 3$, and
\beql{EqK3bis}
\Ht(\K(3)) ~=~
\begin{cases}
5  &\text{if $b=2$}, \\

4,   &\text{if $b=3$}, \\
4,  &\text{if $b \ge 4$ is even}, \\
3,   &\text{if $b \ge 5$ is odd}. \\
\end{cases}
\eeq


\section*{Acknowledgments}
The authors are grateful to the late Donovan Johnson and Reinhard Zumkeller for 
carrying out computations during the early stages of this investigation.




\begin{thebibliography}{10}

\bibitem{Agr14}
N.~Agronomof.
\newblock {Problem 4421}.
\newblock {\em L'Interm\'{e}diaire des math\'{e}maticiens}, 21:147, 1914.

\bibitem{AS99}
J.-P. Allouche and J.~Shallit.
\newblock {The ubiquitous Prouhet--Thue--Morse sequence}.
\newblock In C.~Ding, T.~Helleseth, and H.~Niederreiter, editors, {\em
  Sequences and Their Applications: Proceedings of SETA '98}, pages 1--16.
  Springer, 1999.

\bibitem{Gar88}
M.~Gardner.
\newblock {\em {Time Travel and Other Mathematical Bewilderments}}, chapter 9,
  Section 6, Self-numbers, pages 115--117, 122.
\newblock W.~H.~Freeman, New York, 1988.
\newblock Based on his {\em Mathematical Games} columns in the {\em Scientific
  American} for March and April, 1975.

\bibitem{JMV2007}
J.~Jany\u{s}ka, M.~Modugno, and R.~Vitolo.
\newblock An algebraic approach to physical scales.
\newblock {\em Acta Appl. Math.}, 110:1249--1276, 2010.

\bibitem{Kap51}
D.~R. Kaprekar.
\newblock {Cycles of recurring decimals (from $N=3$ to $161$ and some other
  numbers)}.
\newblock Published by the author, 311 Devlali Camp, Devlali, India, 1951.

\bibitem{Kap56}
D.~R. Kaprekar.
\newblock Self-numbers.
\newblock {\em Scripta Math.}, 22:80--81, 1956.

\bibitem{Kap59}
D.~R. Kaprekar.
\newblock Puzzles of the self-numbers.
\newblock Published by the author, 311 Devlali Camp, Devlali, India, 1959.

\bibitem{Kap62}
D.~R. Kaprekar.
\newblock Self numbers and numbers with 3 and 4 generators.
\newblock Published by the author, 311 Devlali Camp, Devlali, India, 1962.

\bibitem{Kap63}
D.~R. Kaprekar.
\newblock The mathematics of the new self-numbers.
\newblock Published by the author, 311 Devlali Camp, Devlali, India, 1963.

\bibitem{Kap67}
D.~R. Kaprekar.
\newblock {The mathematics of the new self-numbers (Part V)}.
\newblock Published by the author, 311 Devlali Camp, Devlali, India, 1967.

\bibitem{NRao66}
A.~Narasinga~Rao.
\newblock On a technique for obtaining numbers with a multiplicity of
  generators.
\newblock {\em Math. Student}, 34:79--84, 1966.

\bibitem{OEIS}
{OEIS Foundation Inc.}
\newblock {The On-Line Encyclopedia of Integer Sequences}.
\newblock Published electronically at \url{https://oeis.org}.
\newblock Accessed: 2022-01-01.

\bibitem{Rec73}
B.~Recam\'{a}n~Santos (proposer) and D.~W.~Bragg (solver).
\newblock {Problem E2408}.
\newblock {\em Amer. Math. Monthly}, 80, 81:434, 407, 1973, 1974.

\bibitem{Sch04}
R.~Schorn.
\newblock {Kaprekar's Sequence and his ``Selfnumbers''}.
\newblock {\em DERIVE Newsletter}, 53:30--32, 2004.

\bibitem{Sto76}
K.~B. Stolarsky.
\newblock The sum of a digitaddition series.
\newblock {\em Proc. Amer. Math. Soc.}, 59:1--5, 1976.

\end{thebibliography}



\appendix
\renewcommand{\thesection}{\appendixname~\Alph{section}}


\section{Practical Computations}\label{AppPARI}

The most difficult task in programming the algorithm in Section \ref{SecBb} 
is in handling the very large numbers that appear.
As  we saw in Section \ref{SecTow}, to compute
$\K(100)$ in base $10$,  for example, we have to work
with numbers that are of the order of
a tower of $10$'s of height $9$.
This problem was solved using C++ by defining a special 
type of object (we called it the ``sparse radix representation''),
which represents an integer in the algebraic form:
\beql{EqAlgForm}
\frac{1}{\gamma}\left( \al_1 b^{d_1} + \al_2 b^{d_2} +  
\cdots + \al_k b^{d_k}\right),
\eeq
where $\gamma$ and the $\al_i$ are integers,
with $\gamma \geq 1$ (typically we have $\gamma=1$ or $\gamma=b-1$) and $1 \leq \al_i < b$, 
and $d_i$ are objects of the same type, satisfying
$d_1 > d_2 >  \cdots  > d_k$.
These objects support the operations of comparison and addition, 
as well as multiplication by positive rational numbers (so these 
objects form a semi-vector space \cite{JMV2007} 
over the semi-field of positive rationals $\mathbb{Q}^+$).

Comparison of two objects with $\gamma=1$ is done recursively,
starting by comparing the highest-order terms, and if these are tied,
comparing the next-to-highest terms, and so on.
Similarly, addition of two objects with $\gamma=1$ is done
by first combining powers of $b$ with equal exponents,
and then reducing the coefficients into the required range
(e.g., $\beta b^d$ with $\beta \geq b$ is replaced with
$(\beta \bmod b)\, b^{d + \Floor{ \beta/b }}$,
where the addition in the exponents is performed recursively).
If the denominator $\gamma$ of either of the two numbers is not $1$,
the objects (i.e., the coefficients $\al_i$) are
first multiplied by $L$, the least common multiple of their $\gamma$'s,
the coefficients are then reduced into the required range,
and the resulting objects are compared or added as above
(in the case of addition we set $\gamma$ for the sum equal to $L$).
We remark that although the representation of an integer in the
form~\eqref{EqAlgForm} is not unique (e.g., $\frac{1}{b-1}\left(b^1 + (b-2)b^0\right)$ and $2b^0$ represent the same integer $2$), any two such representations
can be efficiently compared and tested for equality.

We also made extensive use of the following
PARI/GP program for computing $\Gen(u)$ and $F(u)$.
The procedure \texttt{Gen(u,b)}
uses the recurrence in \eqref{EqGen}
to compute $\Gen(u)$ in base $b$ for $u \in \NN$. For example,
\texttt{Gen(10\^{}13+1,10)} would return the 
three generators $9999999999892$, $9999999999901$, 
$10000000000000$ of $10^{13}+1$ in base $10$. 
Correspondingly, the number of generators may be obtained as \texttt{\#Gen(u,b)}.

\begin{verbatim}
/* The PARI/GP procedure Gen(u,b) */

{ Gen(u,b=10) = my(d,m,k);
  if(u<0 || u==1, return([]); );
  if(u==0, return([0]); );

  d = #digits(u,b)-1;
  m = u\b^d;
  while( sumdigits(m,b) > u - m*b^d,
    m--;
    if(m==0, m=b-1; d--; );
  );
  k = u - m*b^d - sumdigits(m,b);

  vecsort( concat( apply(x->x+m*b^d,Gen(k,b)), 
                   apply(x->m*b^d-1-x,Gen((b-1)*d-k-2,b)) ) );
}
\end{verbatim}


\section{Flow-charts of computations for 
bases \texorpdfstring{$2$}{2} through \texorpdfstring{$10$}{10}}\label{AppFlow}

To illustrate how the computation of the $\K_i(n)$ and $\B(n)$ proceeds,
and to show the complexity of their recursive structure,
in Figures \ref{fig:b2}--\ref{fig:b10} we provide flow-charts of the 
computation for bases $b=2,3,\dots,10$. 

For each $n\leq 16$, there is a stack of boxed nodes containing 
the values of $\text{E} := \B(n)$ and $\text{K}i := \K_i(n)$ for 
$i\in\sI$. The shaded node denotes the value 
of $\K(n)$ (having the minimum value among all K-nodes in the stack).

Each arc between K-nodes corresponds to an instance 
of the formula~\eqref{EqS2} and 
connects the K-nodes corresponding to $\K_{i-2c_{i,n}}(h_{i,n})$ 
and $\K_i(n)$. 
So, for $n>1$, each $\K_i(n)$ node has exactly one incoming arc.
The value of $\K_i(n)$ is given in the form 
$c_{i,n}(b^{\text{E}}+1)+\ldots$, where ``$\ldots$'' 
stands for the corresponding value of $\K_{i-2c_{i,n}}(h_{i,n})$ 
(located at the starting node of the incoming arc).

Each E-node for $n>1$ has two incoming arcs, shown dashed 
(which may coincide and be shown in bold when $n$ is even),
illustrating the formula
$$
\B(n) = 
\frac{\K_j\left(\Ceil{\tfrac{n}{2}}\right) + 
\K_{-j-2}\left(\Floor{\tfrac{n}{2}}\right)+2}{b-1}\,,
$$
which holds for some $j\in\sI$ as follows from~\eqref{EqFP1}. 
The incoming arcs start at the K-nodes corresponding to 
$\K_j\left(\Ceil{\tfrac{n}{2}}\right)$ and 
$\K_{-j-2}\left(\Floor{\tfrac{n}{2}}\right)$.
Since there may be several equally good choices for $j$, 
we assume that $j$ is such that $\K_j\left(\Ceil{\tfrac{n}{2}}\right) = \K\left(\Ceil{\tfrac{n}{2}}\right)$
whenever equality holds in~\eqref{EqRed}; 
otherwise $j$ is taken to be the smallest of the possible values.
So if equality holds in~\eqref{EqRed}, 
there exists an incoming arc from the shaded K-node with the 
value $\K\left(\Ceil{\tfrac{n}{2}}\right)$.

It can be seen that for $b=4$ (Fig.~\ref{fig:b4}) and 
$b=7$ (Fig.~\ref{fig:b7}) there are no incoming arcs
to the E-node from shaded K-nodes when $n\in\{13,15,16\}$. 
Similarly, there no such arcs for $b=9$ (Fig.~\ref{fig:b9}) 
and $n\in\{8,9,10,11,12,13,14,16\}$.
That is, for these $b$ and $n$, the formula~\eqref{EqRed} does not hold.


 \begin{figure}
  \centering
  \includegraphics[width=.85\textwidth]{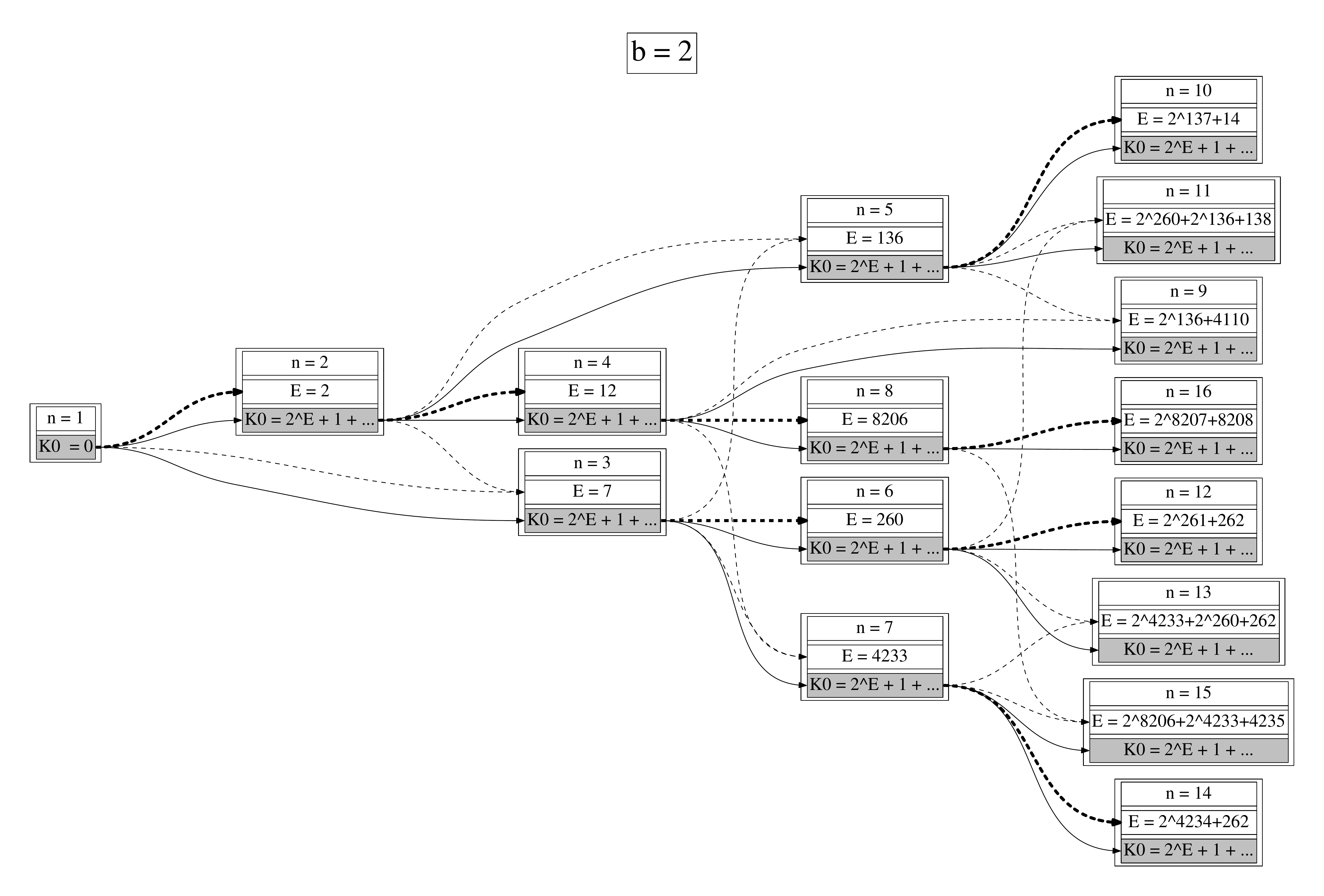}
  \caption{Flow-chart illustrating the calculation of $\B(n)$, $\K_i(n)$, 
  and $\K(n)$ for $n\leq 16$ in base $b=2$. For a description of these flow-charts
  see~\ref{AppFlow}.}
  \label{fig:b2}
 \end{figure}

 \begin{figure}
  \centering
  \includegraphics[width=.85\textwidth]{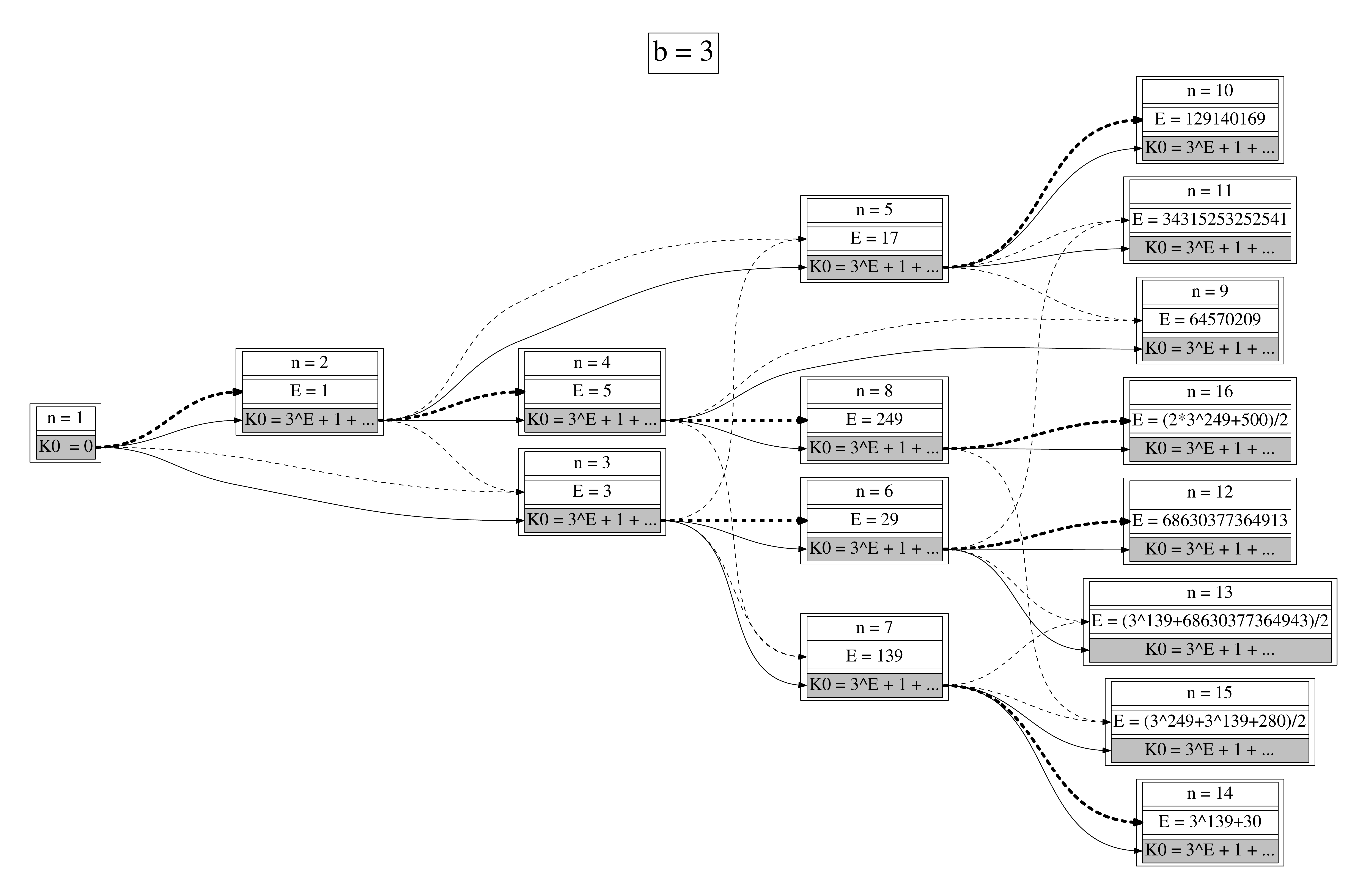}
  \caption{Flow-chart illustrating the calculation of $\B(n)$, $\K_i(n)$, and
  $\K(n)$ for $n\leq 16$ in base $b=3$. 
  Note that apart from the labels, this is the same flow-chart as in Fig.~\ref{fig:b2}.}
  \label{fig:b3}
 \end{figure}

 \begin{figure}
  \centering
  \includegraphics[width=\textwidth]{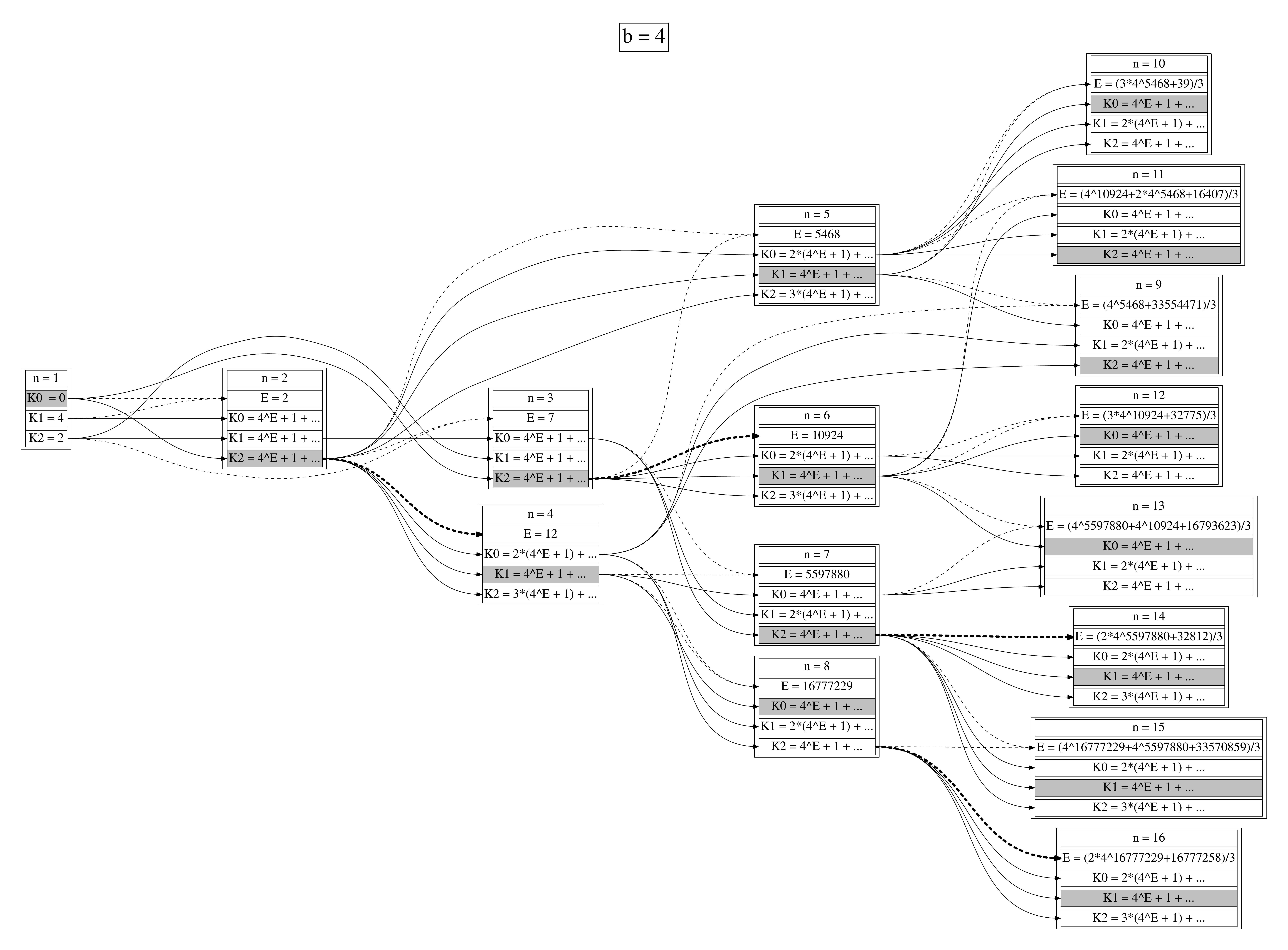}
  \caption{Flow-chart illustrating the calculation of $\B(n)$, $\K_i(n)$, and $\K(n)$ for $n\leq 16$ in base $b=4$.}
  \label{fig:b4}
 \end{figure}

 \begin{figure}
  \centering
  \includegraphics[width=\textwidth]{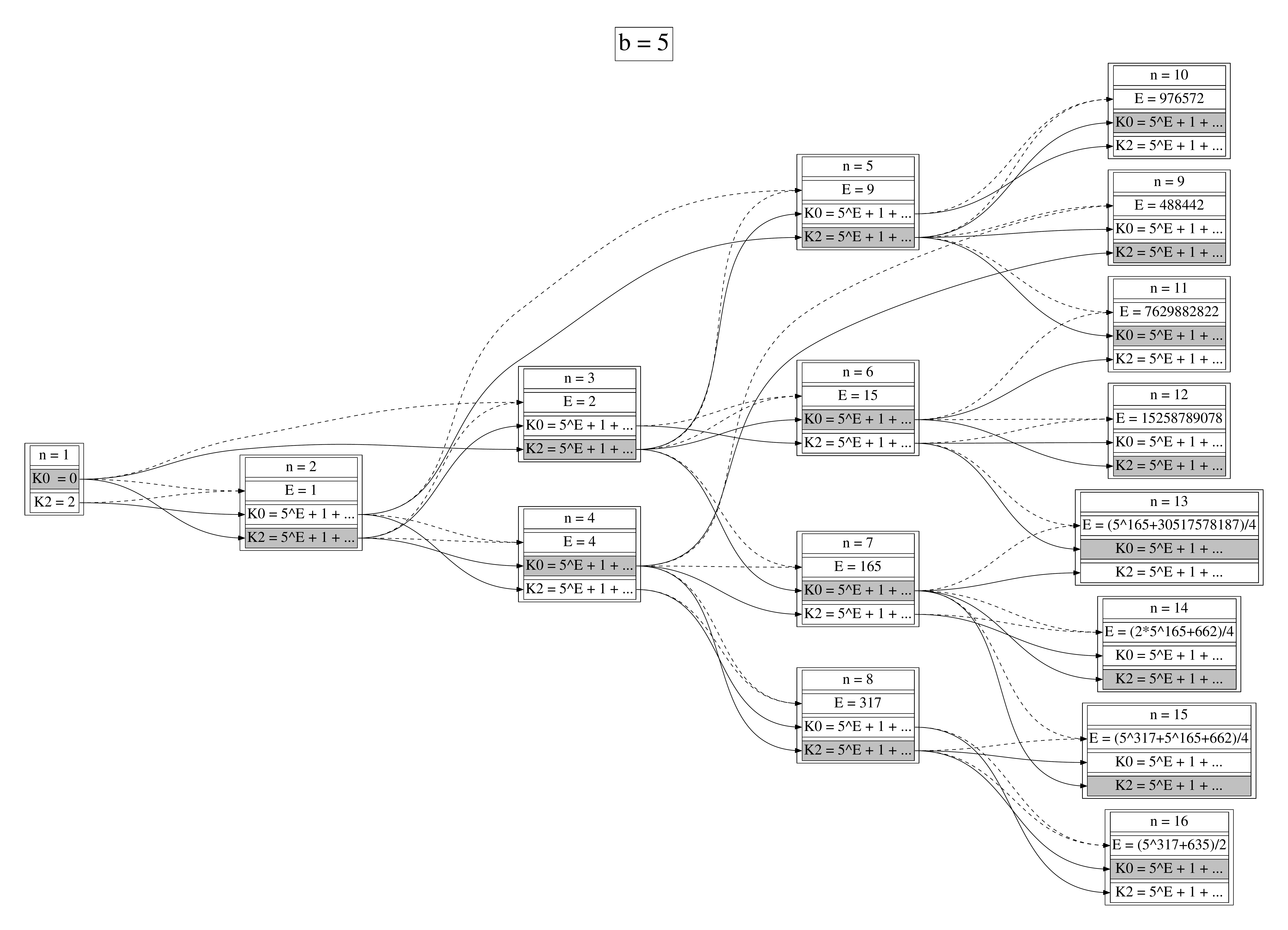}
  \caption{Flow-chart illustrating the calculation of $\B(n)$, $\K_i(n)$, and $\K(n)$ for $n\leq 16$ in base $b=5$.}
  \label{fig:b5}
 \end{figure}

 \begin{figure}
  \centering
 \includegraphics[width=\textwidth]{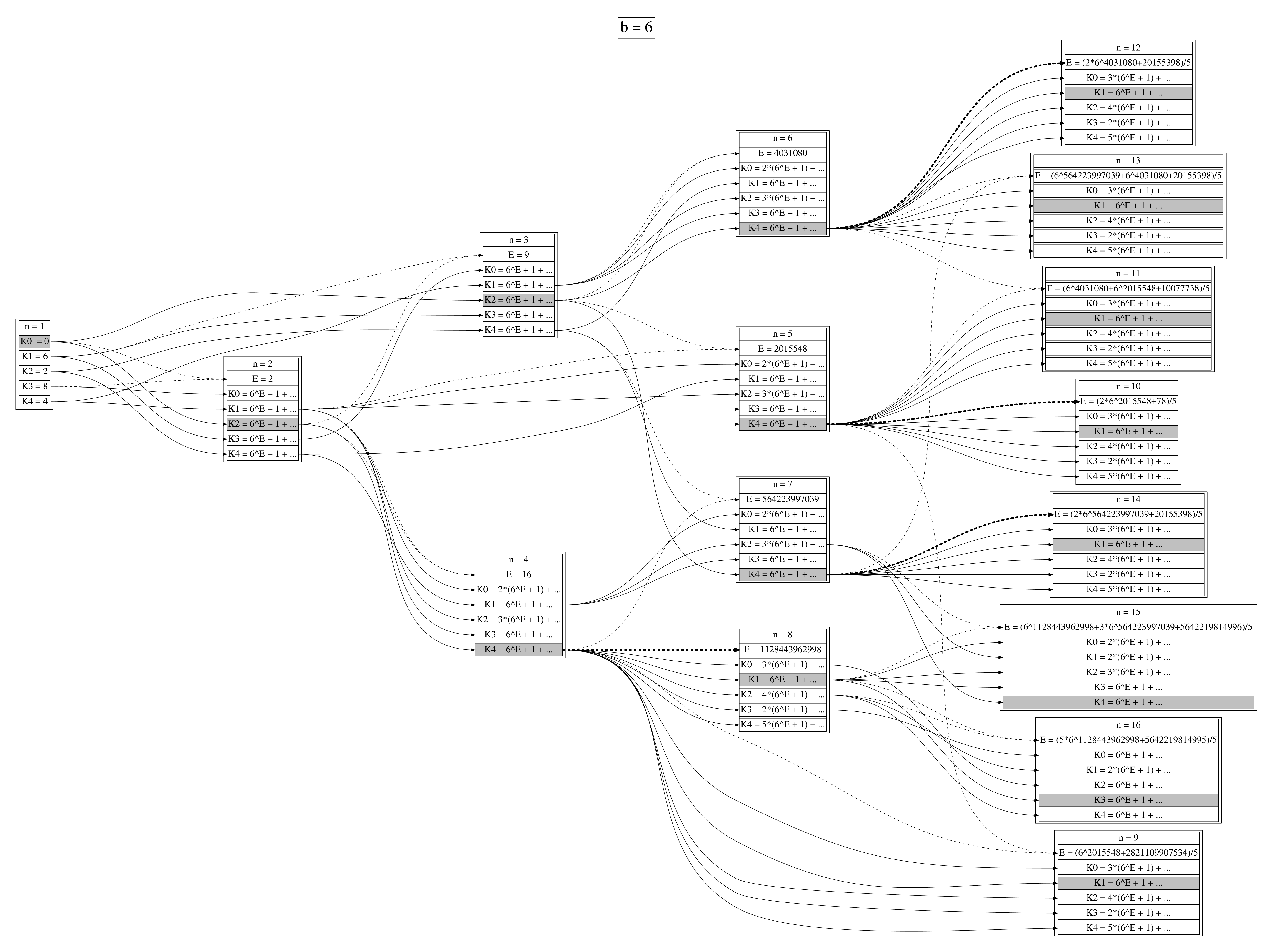}
  \caption{Flow-chart illustrating the calculation of $\B(n)$, $\K_i(n)$, and $\K(n)$ for $n\leq 16$ in base $b=6$. }
  \label{fig:b6}
 \end{figure}

 \begin{figure}
  \centering
  \includegraphics[width=\textwidth]{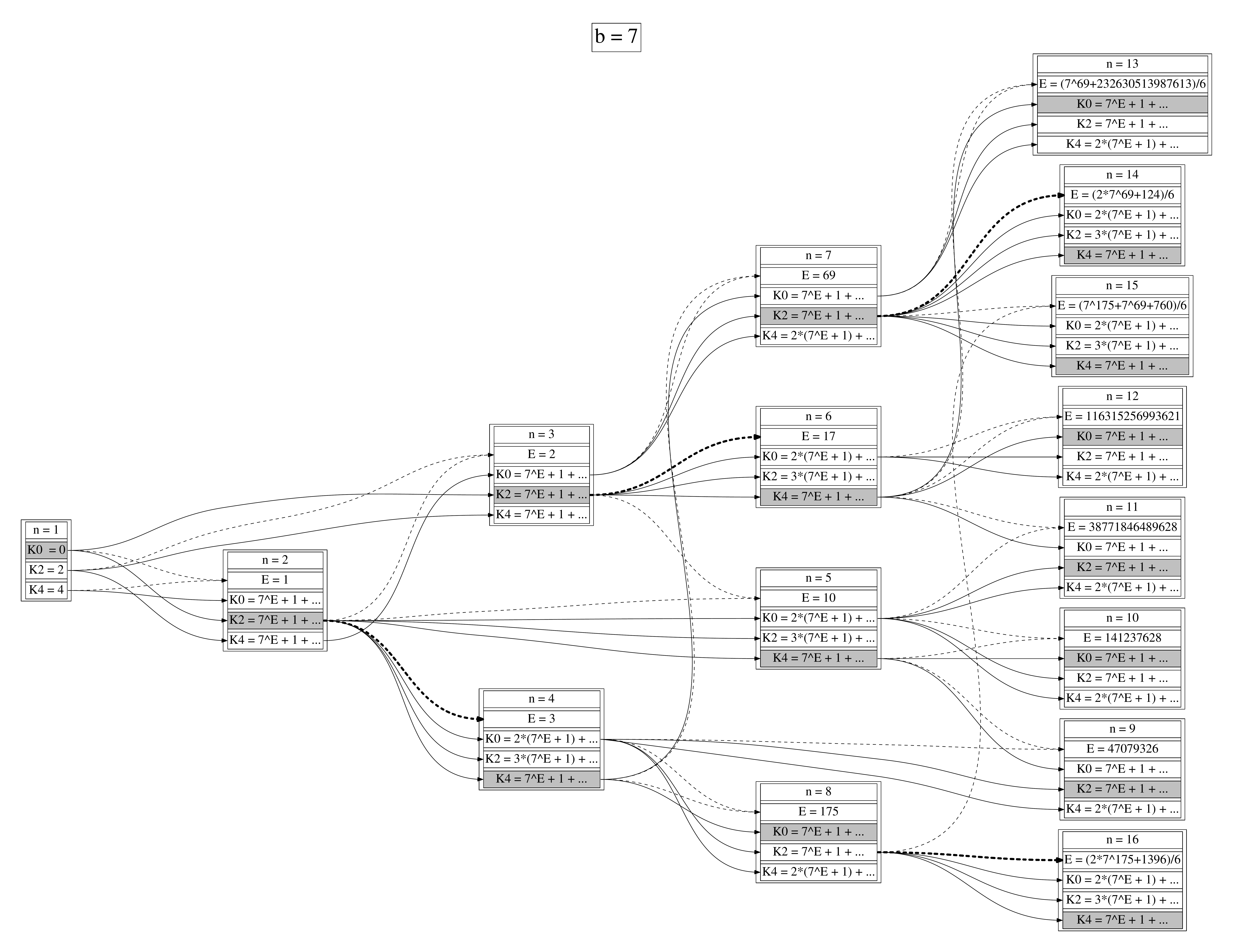}
  \caption{Flow-chart illustrating the calculation of $\B(n)$, $\K_i(n)$, and $\K(n)$ 
  for $n\leq 16$ in base $b=7$. 
  Note that apart from the labels, this is the same flow-chart as in Fig.~\ref{fig:b4}.}
  \label{fig:b7}
 \end{figure}

 \begin{figure}
  \centering
  \includegraphics[width=\textwidth]{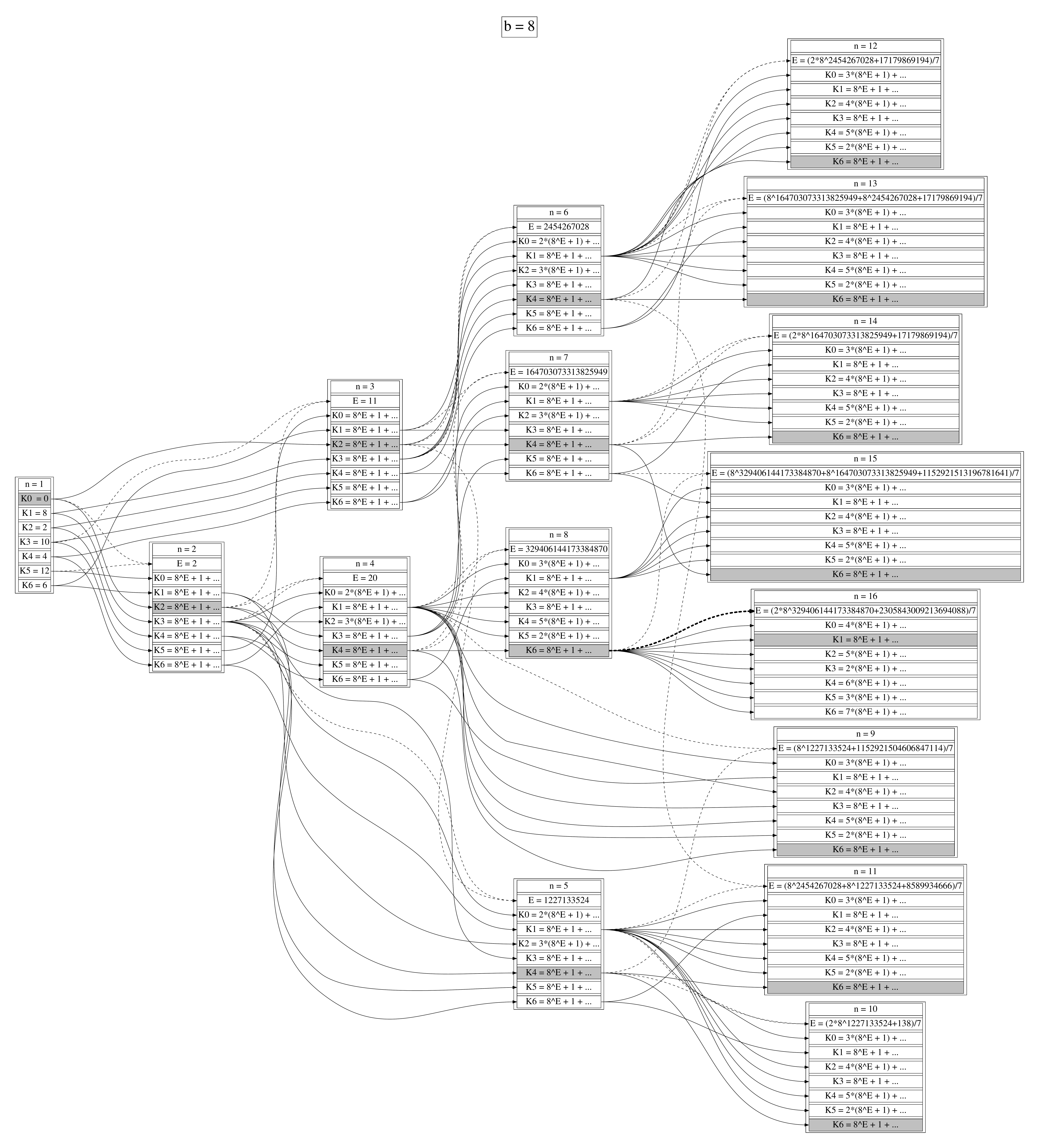}
  \caption{Flow-chart illustrating the calculation of $\B(n)$, $\K_i(n)$, and $\K(n)$ for $n\leq 16$ in base $b=8$. }
  \label{fig:b8}
 \end{figure}

 \begin{figure}
  \centering
  \includegraphics[width=\textwidth]{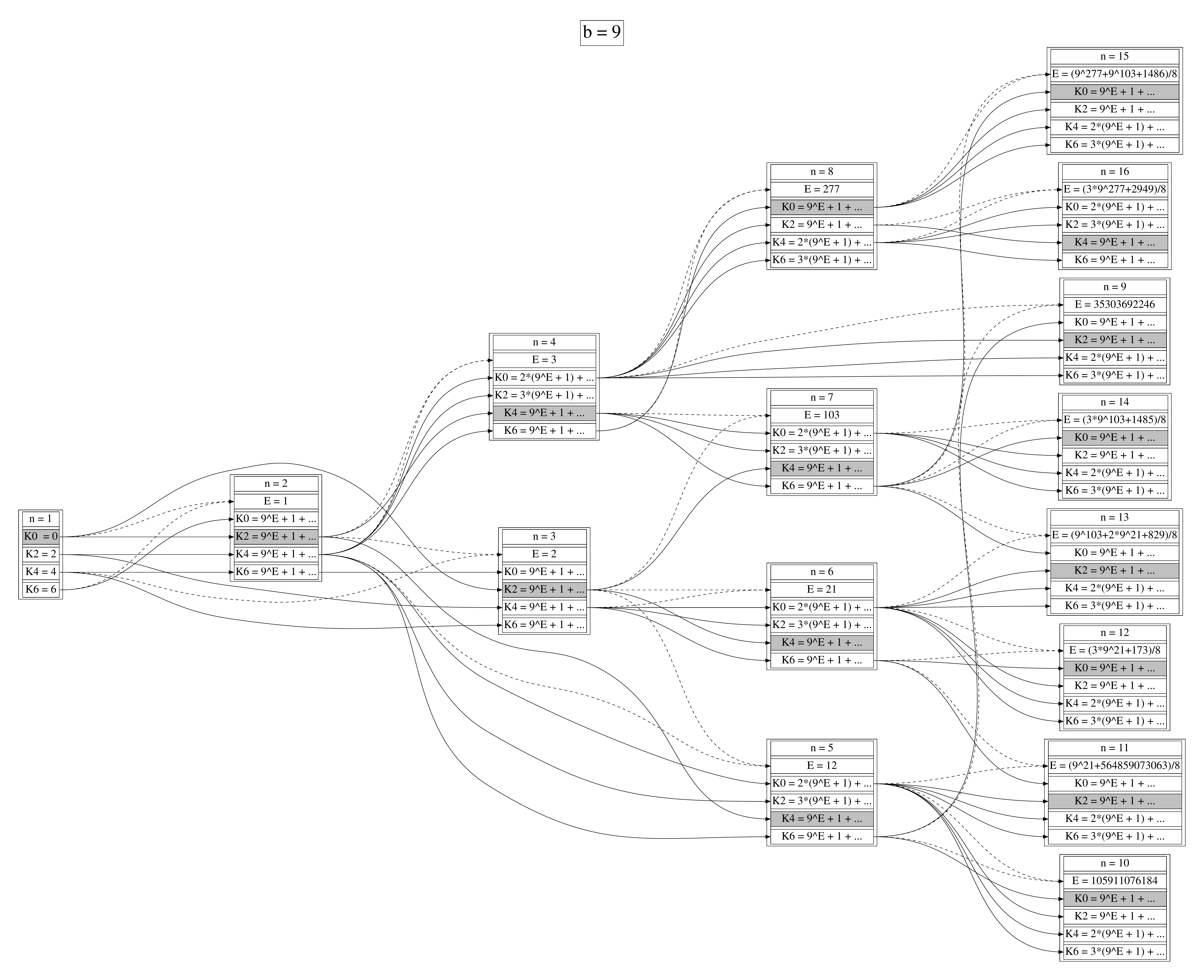}
  \caption{Flow-chart illustrating the calculation of $\B(n)$, $\K_i(n)$, and $\K(n)$ for $n\leq 16$ in base $b=9$. }
  \label{fig:b9}
 \end{figure}

 \begin{figure}
  \centering
  \includegraphics[width=\textwidth]{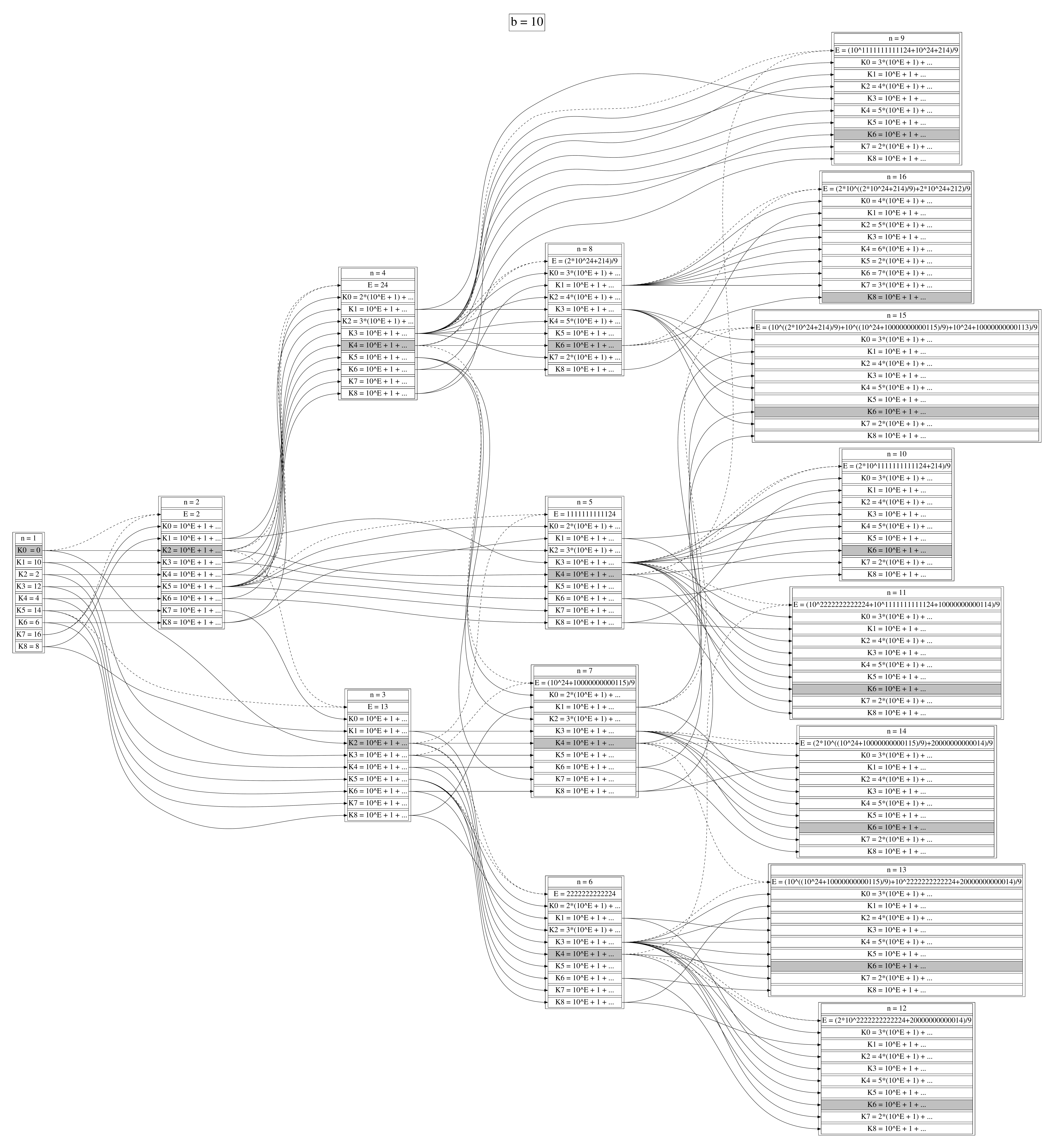}
  \caption{Flow-chart illustrating the calculation of $\B(n)$, $\K_i(n)$, and $\K(n)$ for $n\leq 16$ in base $b=10$. }
  \label{fig:b10}
 \end{figure}

\end{document}